\documentclass[12pt]{amsart}
\usepackage{latexsym,amsmath, amscd, amssymb, amsthm, euscript, mathrsfs}
\usepackage[dvipsnames]{xcolor}
\usepackage[all]{xy}
\usepackage{hyperref}
\usepackage{appendix}
\usepackage[colorinlistoftodos]{todonotes}
\usepackage[cal=boondoxo, scr=boondoxo]{mathalfa}
\usepackage{tikz-cd}
\usepackage{graphicx}
\graphicspath{ {./images/} }

\newtheorem{thm}[subsection]{Theorem}

\newtheorem{thm/def}[subsection]{Theorem/Definition}

\newtheorem{cor}[subsection]{Corollary}

\newtheorem{lem}[subsection]{Lemma}

\newtheorem{prop}[subsection]{Proposition}

\theoremstyle{definition}

\newtheorem{defn}[subsection]{Definition}

\theoremstyle{definition}

\theoremstyle{definition}

\newtheorem{rem}[subsection]{Remark}

\newtheorem{example}[subsection]{Example}

\numberwithin{equation}{subsection}

\newtheorem{pg}[subsection]{}

\newcommand{\mc}{\mathcal }

\newcommand{\Sp}{\text{\rm Spec}}

\newcommand{\Hom}{\underline {\text{\rm Hom}}}

\newcommand{\mls}{\mathscr}

\newcommand{\sEnd}{\mathscr{E}\!{\it nd}}

\newcommand{\Proj}{\underline{\mathrm{Proj}}}
\newcommand{\Spec}{ \underline {\Sp }}
\newcommand{\Gm}{\mathbf{G}_m}
\newcommand{\Ecan}[1]{\mls E^{\mathrm{can}}}

\newcommand{\QP}{\mls{QP}}

\newcommand{\bmu}{\boldsymbol{\mu}}

\begin{document}

\title{Ample vector bundles and moduli of tame stacks}
\author{Daniel Bragg,  Martin Olsson, Rachel Webb}

\begin{abstract}
We explain how to define an embedding of a tame stack over a noetherian ring into a certain generalization of a weighted projective stack using a notion of ample vector bundle on the stack. 
As applications we construct algebraic moduli stacks of tame stacks equipped with an ample vector bundle and algebraic stacks of tame orbicurves.
\end{abstract}

\maketitle

\section{Introduction}

\subsection{Overview} The study of embeddings in, and more generally maps to, projective spaces is fundamental to our understanding of algebraic varieties.  Our aim in this article is to explain analogous constructions for tame stacks.   There are two basic questions to be addressed:
\begin{enumerate}
    \item [(i)] What should replace projective spaces in the theory of tame stacks?  
    \item [(ii)] How do we describe embeddings in such stacks in terms of line bundles, or perhaps vector bundles, as in the classical theory for schemes?
\end{enumerate}

An answer to these questions for cyclotomic stack has been given by Abramovich--Hassett \cite{AH} (a stack is \emph{cyclotomic} if its stabilizers are of the form $\bmu_n$). In this case the replacement for projective spaces are the weighted projective stacks
\[
    [\mathbf{A}^{n+1}-\{0\}/_{\underline \alpha }\mathbf{G}_m] \quad \quad \quad \text{where} \quad u \cdot (x_0, \ldots, x_n) = (u^{\alpha _0}x_0, \dots, u^{\alpha _n}x_n)
\]
for $u \in \mathbf{G}_m$ and a sequence of positive integers $\underline{\alpha}=(\alpha_0, \ldots, \alpha_n)$. In \cite[2.1.3 and 2.4.4]{AH} maps to such stacks are discussed in terms of line bundles and sections. We also note work of El Haloui \cite{haloui2016ample}, who considered related notions using graded commutative algebra, and work of Voight and Zureick-Brown \cite{MR4403928} studying the geometry of stacky curves through their canonical rings. 
A weighted projective stack is cyclotomic, as is any substack. Thus, to move beyond cyclotomic stacks, a general answer to (i) must replace weighted projective stacks with a more general class of stacks. Moreover, if $\mc X$ is an algebraic stack and $\mc L$ is a line bundle such that the action of the stabilizer group schemes on the fibers of $\mc L$ is faithful then the stack $\mc X$ is necessarily cyclotomic. It follows that a general answer to (ii) must replace line bundles with vector bundles. 

For Deligne--Mumford stacks over a field, questions (i) and (ii) have been studied by Kresch \cite{MR2483938}. Our approach has something in common with his, but differs in our emphasis on the role of an ample vector bundle and in our construction of the resulting embeddings. We track the data of the bundle and the embedding so that we can we apply our theory to construct moduli of tame stacks equipped with ample vector bundles. In the case of one-dimensional stacks with generically trivial stabilizer, we are able to forget the data of the bundle and construct a moduli stack of tame orbicurves.

\subsection{Ample vector bundles} \label{P:1.2} Let $R$ be a noetherian ring and let $\mls X/R$ be a finite type tame Artin stack with finite diagonal over the base (see \cite{tame} for the notion of tame Artin stack; in characteristic $0$, the stacks considered here are simply finite type separated Deligne--Mumford stacks). Let $\mls E$ be a vector bundle over $\mls X$. For each geometric point $\bar x\rightarrow \mls X$ we have an action of the stabilizer group scheme $G_{\bar x}$ on the fiber $\mls E(\bar x):= \bar x^*\mls E$. Let $\pi: \mls X \to X$ be the coarse moduli space map and let $S^n\mls E$ denote the $n$-th symmetric power of $\mls E$. 

\begin{defn}\label{D:ampledef}  A vector bundle $\mls E$ over $\mls X$ is 

(i)  \emph{faithful} if for every geometric point $\bar x\rightarrow \mls X$ the representation $\mls E(\bar x)$ of $G_{\bar x}$ is faithful.

(ii)  \emph{$H$-ample} if it is faithful and for every coherent sheaf $\mls F$ on $\mls X$ there exists an integer $n_0$ such that $\pi _*(\mls F\otimes S^n\mls E)$ is generated by global sections for all $n\geq n_0$. 

(iii) \emph{det-ample} if it is faithful and for some integer $N>0$ the line bundle $\text{\rm det}(\mls E)^{\otimes N}$ descends to an ample invertible sheaf on the coarse space $X$.
\end{defn}

\begin{rem} The $H$ in  ``$H$-ample'' stands for Hartshorne, in reference to the definition of ample vector bundle in \cite{Hartshorneample}. In the case of schemes an $H$-ample vector bundle is also $\det$-ample by \cite[2.6]{Hartshorneample}. On stacks an $H$-ample bundle which is also \textit{generating} is det-ample (see \ref{L:3.2}), and for line bundles the two notions of ampleness are equivalent (see \ref{E:cyclotomic}). However in general we don't know if $H$-ample implies det-ample. In the converse direction, if a stack admits a $\det $-ample vector bundle then it also admits a vector bundle which is both $\det $- and $H$-ample \ref{cor:4.6}. 
\end{rem}

\begin{rem}
    It is tempting to replace the condition in (ii) with the condition that $\mls F\otimes S^n\mls E$ is generated by global sections.  However, this does not give a good notion since it implies that the adjunction map $\pi ^*\pi _*(\mls F\otimes S^n\mls E)\rightarrow \mls F\otimes S^n\mls E$ is surjective, which  in turn implies that the stabilizer action on $\mls F\otimes S^n\mls E$ is trivial at every point. Since $\mls X$ is tame this implies that $\mls F\otimes S^n\mls E$ descends to the coarse space.  So, for example, if $\mls E-\mls L$ is a line bundle and $\mls F = \mls O_{\mls X}$, the condition that $\mls L^{\otimes n}$ is generated by global sections for $n$ sufficiently large implies that the stabilizer action on $\mls L$ is trivial and therefore $\mls L$ descends to the coarse space.
\end{rem}





We study these notions of ampleness in detail in Section \ref{S:ample}.


\subsection{Immersions from ample vector bundles}
\label{P:1.9}  Next let us explain the stacks that will play the role of projective spaces.
 Let $R$ be a noetherian ring and 
let $V$ be a finitely generated  $R$-module with a left polynomial $\mathbf{GL}_{r}$-action (defined in \ref{def:polynomial representation}). We have an associated right action of $\mathbf{GL}_r$ on $\mathbf{A}_V := \Sp_R(S^\bullet V)$.  Let $\det:\mathbf{GL}_r\to\mathbf{G}_m$ denote the determinant character. By \ref{T:GIT} there is an associated stable locus $\mathbf{A}^{s,\; \det}_{V}$, open in $\mathbf{A}_{V}$, such that the stack quotient $[\mathbf{A}^{s,\; \det }_{V}/\mathbf{GL}_r]$ has finite diagonal. By \cite[Proposition 3.6]{tame} there is a maximal tame open substack of this quotient.

\begin{defn}\label{def:qp}
We let $\QP(V)$ denote the maximal tame open substack of $[\mathbf{A}^{s,\; \det }_{V}/\mathbf{GL}_r]$. (The letters $\QP$ stand for ``quasi-projective.'')
\end{defn}

By construction, the stack $\QP(V)$ is a finite type tame Artin stack with finite diagonal and quasi-projective coarse moduli space $QP(V)$ (quasi-projectivity of the coarse space follows from \ref{C:B.12}). 
Since $\QP(V)$ is constructed as a $\mathbf{GL}_r$-quotient it carries a canonical faithful rank-$r$ vector bundle, which by  \ref{rem:git-detample} is, in fact,  $\det$-ample.


\begin{rem} 
If the ring $R$ is equal to a field $k$ then $QP(V)$ is an open subscheme of the twisted affine GIT quotient of $\mathbf{A}_V$ by $\mathbf{GL}_r$ with respect to the character $\det$. In particular, if $\mathbf{A}_V^{ss, \det } = \mathbf{A}_V^{s, \det }$, if $k$ has characteristic $0$,  and if the affine quotient equals $\Sp (k)$, then $QP(V)$ and hence also $\QP (V)$ are proper over $R$.

\end{rem}

Our results concerning ample vector bundles and embeddings are summarized by the following (compare with \cite[Thm~5.3]{MR2483938}).

\begin{thm}\label{cor:4.6}
Let $\mls X$ be a finite type tame Artin stack with finite diagonal over a noetherian ring $R$. The following are equivalent:
\begin{enumerate}
\item $\mls X$ is globally the quotient of an $R$-scheme by $\mathbf{GL}_{r, R}$ for some $r$ and has quasi-projective coarse moduli space.
\item $\mls X$ admits a vector bundle that is both $H$-ample and $\det$-ample.
\item $\mls X$ admits a $\det$-ample vector bundle.
\item $\mls X$ admits an immersion into a stack of the form $\QP(V)$.
\end{enumerate}
\end{thm}
The way in which a det-ample vector bundle defines an immersion is often important for applying \ref{cor:4.6}. The construction of this immersion is explained in Section \ref{S:section4}.
A relative version of ample vector bundles and \ref{cor:4.6} are discussed in Section \ref{S:section6}.




\begin{rem} Our construction of the immersion in \ref{cor:4.6} is related to  the argument of \cite[Proof of 5.3]{MR2483938}.
Two key differences are that the construction in \cite[Proof of 5.3]{MR2483938} embeds $\mls X$ into a stack of the form $[\mathbf{P}(V)/\mathbf{GL}_r]$ for some projective space $\mathbf{P}(V)$ with an action of $\mathbf{GL}_r$, whereas our embedding lands in a stack of the form $[\mathbf{A}_V/\mathbf{GL}_r]$; and that we keep track of the spaces of sections defining this embedding.
\end{rem}

\subsection{Moduli of tame stacks}
We apply our theory of ample vector bundles to construct various moduli of tame stacks. 
If $T$ is a scheme, let $\mls S_r(T)$ be the category of tame stacks over $T$ equpped with a rank $r$ det-ample vector bundle: objects of $\mls S_r(T)$ are pairs $(\mls X, \mls E)$ where $\mls X$ is a proper flat tame Artin stack over $T$ and $\mls E$ is a vector bundle of rank $r$ on $\mls X$ that is det-ample in every geometric fiber. Morphisms are isomorphism classes of pairs of isomorphisms $(f: \mls X \xrightarrow{\sim} \mls X', f^*\mls E' \xrightarrow{\sim} \mls E)$. Let $\mls S_r$ denote the fibered category over schemes whose fiber over $T$ is given by $\mls S_r(T)$. In Section \ref{S:stackofstacks} we prove the following.
\begin{thm}[Theorem \ref{T:5.1}]
$\mls S_r$ is an algebraic stack.
\end{thm}

As a second application, let $\mls M(T)$ be the category of tame orbicurves over $T$: objects are flat proper tame Artin stacks $\mls C \to T$ such that geometric fibers are geometrically connected of dimension 1 and have a dense open subscheme. Morphisms are isomorphism classes of isomorphisms of stacks over $T$. Let $\mls M$ be the fibered category over schemes associated to the categories $\mls M(T)$. In Section \ref{S:section8} we show the following.
\begin{thm}[Proposition \ref{P:7.11} and Theorem \ref{T:8.13}]
$\mls M$ is an algebraic stack.
\end{thm}
The geometric objects parametrized by $\mls M$ are curves with arbitrary singularities and stabilizer groups. Hence this stack may be viewed as an analog of the algebraic stack of representable curves in \cite[\href{https://stacks.math.columbia.edu/tag/0D5A}{Tag 0D5A}]{stacks-project} (see also \cite[3.3]{dJHS} \cite[B.1]{smyth}).


The article includes two appendices.  The first contains a few results about ample line bundles on algebraic spaces that we could not locate in the literature.  The second discusses geometric invariant theory over general rings, following work of Alper \cite{Alperadequate}, which in turn generalizes work of Seshadri \cite{Seshadri}.  

\begin{rem} Though we do not address it here, we expect that at least some of the results of this article can be extended to certain stacks with positive-dimensional stabilizer group schemes and adequate moduli spaces.
\end{rem}

\subsection{Conventions}
A \emph{vector bundle} on an algebraic stack is a locally free sheaf of finite  rank (we allow the rank to vary across connected components). If $\mls W$ is a vector bundle on an algebraic stack $\mls X$ then $\mathbf{GL}(\mls W)$ is the group scheme over $\mls X$ associated to automorphisms of $\mls W$ and $\mls End(\mls W)$ is the monoid scheme associated to endomorphisms of $\mls W$.  They act on $\mls W$ on the left. Finally, if $\mls W$ is any coherent sheaf on $\mls X$ we use $S^\bullet \mls W$ to denote the symmetric algebra of $\mls W$ and we define $\mathbf{A}_{\mls W} := \Spec_{\mls X}(S^\bullet \mls W)$. When $\mls W$ is locally free of finite rank we call $\mathbf{A}_{\mls W}$ the associated geometric vector bundle. 

The notion of tame stack is defined in \cite[3.1]{tame}.  Note in particular that all tame stacks are assumed of finite presentation over a specified base scheme.  We will further assume throughout that all tame stacks have finite diagonal over the base. 

If $\mls X$ is a tame stack over a scheme $S$ then it has a coarse moduli space $\pi :\mls X\rightarrow X$, with $X$ an algebraic space of finite presentation over $S$.  In the case when $S$ is locally noetherian this is shown in \cite{KeelMori}.  The general case is treated in \cite{KMConrad}.  Because of the tameness assumption, which implies that the formation of coarse moduli space commutes with base change,  the coarse space can be described as follows.  If $\Sp (R)\subset S$ is an affine open, then since $\mls X/S$ is of finite presentation there exists a finitely generated (in particular noetherian) subring $R_0\subset R$ such that $\mls X|_{\Sp (R)} = \mls X_0\times _{\Sp (R_0)}\Sp (R)$ for a tame stack $\mls X_0/R_0$.  The restriction $X|_{\Sp (R)}$ is then simply $X_0\times _{\Sp (R_0)}\Sp (R)$, where $X_0$ is the coarse space of $\mls X_0$.

\subsection{Acknowledgements} Bragg received support from NSF grant 1840190. Olsson was partially supported by NSF FRG grant DMS-2151946 and the Simons Foundation. Webb was partially supported by a grant from the Simons Foundation, and also benefited from the hospitality and stimulating academic environment of the Isaac Newton Institute. 
 This work originated with discussions at a meeting of the AMS Mathematics Research Community  \emph{Explicit Computations with Stacks} supported by the NSF under Grant Number DMS 1916439, whose stimulating environment is gratefully acknowledged.    Separately the authors learned that Max Lieblich has considered similar questions in unpublished work.

\section{Polynomial representations and stacks}\label{S:section2}

\subsection{Motivation: polynomial representations of $\Gm$}
To motivate our discussion of polynomial representations, we recall two characterizations of maps to projective space. 

If $\mls L$ is a line bundle on a scheme $X$ let $\mathbf{A}_{\mls L} := \Spec_X(\oplus_{m \in \mathbf{N}} \mls L^{\otimes m})$ be the associated geometric line bundle equipped with the scaling $\Gm$ action. There are functorial bijections
\begin{equation}\label{eq:motivate2}
\left\{\begin{array}{c}
\text{morphisms}\\
X \to \mathbf{P}^n \end{array}\right\} \leftrightarrow \left\{\begin{array}{c}
\text{line bundles}\;\mls L\;\text{and}\\
\text{sections}\;s_0, \ldots, s_n\;\text{of}\;\mls L\\
\text{with no common zero}\end{array}\right\} \leftrightarrow \left\{\begin{array}{c}
\text{geometric line bundles}\;\mathbf{A}_{\mls L}\;\text{and}\\
\text{equivariant morphisms}\;\\
\mathbf{L} \to \mathbf{A}^{n+1}\setminus \{0\}\end{array}\right\}
\end{equation}
where $\Gm$ acts on $\mathbf{A}^{n+1} \setminus \{0\}$ by scaling.
Indeed, a choice of sections $s_i \in H^0(X, \mls L)$ for $i=0, \ldots, n$ is equivalent to a graded ring homomorphism 
\[
k[x_0, \ldots, x_n] \to \oplus_{m \geq 0} H^0(X, \mls L^{\otimes m});
\]
i.e., an equivariant morphism $\mathbf{A}_{\mls L} \to \mathbf{A}^{n+1}$.

If we replace the scaling $\Gm$ action on $\mathbf{A}^{n+1}$ with any action having integer weights $\alpha_0, \ldots, \alpha_n$, we obtain a description similar to \eqref{eq:motivate2} for morphisms to the stack $\mls P := [\mathbf{A}^{n+1}\setminus\{0\}/\Gm]$:
\[
\left\{\begin{array}{c}
\text{morphisms}\\
X \to\mls P \end{array}\right\} \leftrightarrow \left\{\begin{array}{c}
\text{line bundles}\;\mls L\;\text{and}\\
\text{sections}\; s_i \in H^0(X, \mls L^{\otimes a_i})\\
\text{with no common zero}\end{array}\right\} \leftrightarrow \left\{\begin{array}{c}
\text{geometric line bundles}\;\mathbf{A}_{\mls L}\;\text{and}\\
\text{equivariant morphisms}\;\\
\mathbf{A}_{\mls L}^* \to \mathbf{A}^{n+1}\setminus \{0\}\end{array}\right\}
\]
Here, $\mathbf{A}_{\mls L}^* \subset \mathbf{A}_{\mls L}$ is the complement of the zero section.
The key to the second bijection is that our line bundle $\mls L$ and sections $s_i \in H^0(X, \mls L^{\otimes \alpha_i})$ now define a graded ring homomorphism 
\begin{equation}\label{eq:motivate1}
k[x_0, \ldots, x_n] \to \oplus_{m \in \mathbf{Z}} H^0(X, \mls L^{\otimes m}).
\end{equation}
The morphism $\mathbf{A}_{\mls L}^* \to \mathbf{A}^{n+1}$ extends to a morphism $\mathbf{A}_{\mls L} \to \mathbf{A}^{n+1}$ exactly when the ring map \eqref{eq:motivate1} factors through $\oplus_{m \geq 0} H^0(X, \mls L^{\otimes m})$; i.e., exactly when the weights $\alpha_i$ are nonnegative.\footnote{When the weights are in fact all positive, the stack $\mls P$ is the stacky weighted projective space with weights $\alpha_0, \ldots, \alpha_n$.} Thus if we want a characterization of morphisms to $\mls P$ that fully generalizes the bijections in \eqref{eq:motivate2}, we must require $\alpha_i \geq 0$.

The representations $\mathbf{A}^{n+1}$ of $\Gm$ with all nonnegative weights are precisely the \textit{polynomial} representations of $\Gm$. Polynomiality of representations of $\mathbf{GL}_r$ and the bijections analogous to \eqref{eq:motivate2} are discussed in the next sections. 

\begin{example}

Let $\mathbf{G}_m$ act on $\mathbf{A}^2$ with weights $(1, -1)$ and let $\mls P$ denote the corresponding weighted projective stack. For any stack $\mls X$, a morphism $\mls X \to \mls P$ is equivalent to a line bundle $\mls L$ on $\mls X$ and an equivariant morphism $\mathbf{A}^*_{\mls L} \to \mathbf{A}^2\setminus\{0\}$, or equivalently a $\mathbf{G}_m$-equivariant ring map 
    $$
    \underline {\Sp }_{\mls X}(\oplus _{n\in \mathbf{Z}}\mls L)\rightarrow \mathbf{A}^2.
    $$
    that is, a morphism $\mls X \to \mls P$ is equivalent to a line bundle $\mls L$ and sections
    $s\in H^0(\mls X, \mls L)$ and $t\in H^0(\mls X, \mls L^{-1})$.  Note, in particular, that the induced map $\mls O_{\mls X}^{\oplus 2}\rightarrow \oplus _{n\in \mathbf{Z}}\mls L^{\otimes n}$ does not have image in $\oplus _{n\geq 0}\mls L^{\otimes n}.$
\end{example}

\subsection{Relative representations of $\mathbf{GL}_r$}

Let $\mls Y$ be an algebraic stack over a ring $R$. Let $R[x_{ij}]_{\det}$ be the coordinate ring of $\mathbf{GL}_{r, R}$, where $1 \leq i, j \leq r$. A representation of $\mathbf{GL}_r$ over $\mls Y$ is a vector bundle $\mls W$ on $\mls Y$ together with a group homomorphism 
\begin{equation}\label{eq:rep1}
\mathbf{GL}_{r, \mls Y} \to \mathbf{GL}(\mls W)
\end{equation}
over $\mls Y$. Such a homomorphism corresponds to a left action of $\mathbf{GL}_{r, \mls Y}$ on $\mls W$, or geometrically a right action
\begin{equation}\label{eq:rep2}
\mathbf{A}_{\mls W} \times \mathbf{GL}_r \to \mathbf{A}_{\mls W}
\end{equation}
satisfying certain axioms. It is also equivalent to the data of a $R[x_{ij}]_{\det}$-comodule; that is, an $\mls O_{\mls Y}$-module homomorphism
\begin{equation}\label{eq:rep3}
\mls W \to \mls W \otimes_R R[x_{ij}]_{\det}
\end{equation}
satisfying certain axioms. For more on these equivalences see for example \cite[4.a]{Milne}.

Let $M_r$ denote the monoid scheme of $r \times r$ matrices, so we have an open  immersion $\mathbf{GL}_{r, R} \to M_{r, R}$ dual to the localization map
\[
R[x_{ij}] \to R[x_{ij}]_{\det}.
\]
\begin{defn}\label{def:polynomial representation}
    Following classical terminology (see for example \cite{Green}) we say that the representation $\mls W$ is \textit{polynomial} if the associated coaction morphism \eqref{eq:rep3} has image in $\mls W \otimes_R R[x_{ij}]$. Equivalently, the homomorphism \eqref{eq:rep1} extends to a morphism of monoid schemes $M_{r, \mls Y} \to  \mls End(\mls W)$ over $\mls Y$, and the associated $\mathbf{GL}_r$-module \eqref{eq:rep2} is the restriction of an action of the monoid scheme $M_r$ on $\mathbf{A}_{\mls W}$.
\end{defn}

\begin{defn} We say that a polynomial representation $\mls W$ has \emph{degree $\leq M$} if the image of the coaction map $\mls W\rightarrow \mls W\otimes _RR[x_{ij}]$ is contained in the sub-$R$-module $\mls W\otimes _RR[x_{ij}]^{\leq m}$, where $R[x_{ij}]^{\leq m}\subset R[x_{ij}]$ is the submodule of polynomials of degree $\leq m$.
\end{defn}

\begin{rem} Concretely if $\mls W$ is a polynomial representation, then the condition that $\mls W$ has degree $\leq M$ is equivalent to the following.  Locally on $\Sp (R)$ we can trivialize $\mls W$ in which case the action of $M_r$ on $\mls W$ is given by a map $t:M_r\rightarrow M_s$, where $s$ is the rank of $\mls W$.  Therefore for any $R$-algebra $A$ and matrix $U\in M_r(A)$ we get another matrix $t(U)\in M_s(A)$.  The condition that $\mls W$ has degree $\leq M$ means that the entries of $t(U)$ are given by polynomials of degree $\leq M$ in the entries of $U$.  
\end{rem}

\begin{rem} Some discussion of monoid schemes can be found in \cite[Expos\'e I, 2.2]{SGA31}. Concretely a monoid scheme $M$ over a stack $\mls Y$ consists of a representable morphism $M\rightarrow \mls Y$ and a map $M\times _{\mls Y}M\rightarrow M$ satisfying the usual axioms for a monoid.
\end{rem}

\begin{rem}\label{R:relative-actions}
More generally, if $\mls W$ is a coherent sheaf on $\mls Y$ an action of $\mathbf{GL}_r$ on $\mls W$ is a group homomorphism $\mathbf{GL}_r \to \mathrm{Aut}(\mls W)$ (where automorphisms of $\mls W$ act on the left). This data is equivalent to a geometric right action $\mathbf{A}_{\mls W} \times \mathbf{GL}_r \to \mathbf{A}_{\mls W}$ over $\mls Y$ on $\mathbf{A}_{\mls W}:= \Spec _{\mls Y}(\oplus _{n\geq 0}S^n\mls W)$ associated to $\mls W$, and it is also equivalent to a comodule structure $\mls W \to \mls W \otimes_R R[x_{ij}]_{\det}$.
\end{rem}

\subsection{Morphisms to quotients of polynomial representations}
Let $\mls Y$ be an algebraic stack and let $\mls W$ be a representation of $\mathbf{GL}_r$ over $\mls Y$. We have an associated algebraic stack $\mls A_{\mls W} := [\mathbf{A}_\mls W / \mathbf{GL}_{r, \mls Y} ]$ over $\mls Y$.

\begin{thm}\label{P:2.3} Let $\mls W$ be a polynomial representation of $\mathbf{GL}_r$ over an algebraic stack $\mls Y$. Let $f:\mls X\rightarrow \mls Y$ be a morphism of algebraic stacks.  Then the groupoid of liftings of $f$ to a morphism $\tilde f:\mls X\rightarrow \mls A_{\mls W}$ is equivalent to the groupoid of data $(\mls F, \gamma )$ as follows:
\begin{enumerate}
    \item [(i)] $\mls F$ is a vector bundle of rank $r$ on $\mls X$.
    \item [(ii)] $\gamma: \mathbf{A}_{\mls F^{\oplus r}} \to \mathbf{A}_\mls W$ is a $\mathbf{GL}_r$-equivariant morphism, where the right action on $\mathbf{A}_{\mls F^{\oplus r}}$ is induced by the natural structure of $\mls F^{\oplus r}$ as a $\mathbf{GL}_r$-representation. 
\end{enumerate}
\end{thm}
\begin{rem}\label{R:2.8}
The $\mathbf{GL}_r$-equivariant morphism $\gamma$ appearing in (ii) is equivalent to a morphism of $\mls O_{\mls X}$-modules
\begin{equation}\label{eq:2.3-ring}
f^*\mls W \to S^\bullet(\mls F^{\oplus r})
\end{equation}
that is $\mathbf{GL}_r$-equivariant in the appropriate sense.
\end{rem}

\begin{proof}[Proof of \ref{P:2.3}]
From \cite[\href{https://stacks.math.columbia.edu/tag/04UV}{Tag 04UV}]{stacks-project} we know that the groupoid of liftings $\tilde f$ is equivalent to the groupoid of pairs $(\mathbf{I}, g)$ where $p:\mathbf{I} \to \mls X$ is a principal $\mathbf{GL}_r$-bundle (where $\mathbf{GL}_r$ acts on the right) and $g: \mathbf{I} \to \mathbf{A}_{\mls W}$ is an equivariant homomorphism over $\mls Y$. Let $\mls F$ be the vector bundle on $\mls X$ such that $\mathbf{A}_{\mls F} = \mathbf{I} \times^{\mathbf{GL}_r} \mathbf{A}^r$ where $\mathbf{GL}_r$ acts on $\mathbf{A}^r$ in the standard manner. Then $\mathbf{A}_{\mls F^{\oplus r}} = \mathbf{I} \times^{\mathbf{GL}_r} M_r$ and we have the determinant morphism
\begin{equation}\label{eq:det}
\det: \mathbf{A}_{\mls F^{\oplus r}} \to \mathbf{A}_{\det(\mls F)}.
\end{equation}
Moreover, $\mathbf{I}$ is the open substack of $\mathbf{A}_{\mls F^{\oplus r}}$ equal to the preimage under \eqref{eq:det} of the complement of the zero section of $\mathbf{A}_{\det(\mls F)}$. In particular we have an injective homomorphism of $\mls O_\mls X$-algebras $S^\bullet \mls F^{\oplus r} \hookrightarrow p_*\mls O_{\mathbf{I}}$.

The equivariant homomorphism $g: \mathbf{I} \to \mathbf{A}_{\mls W}$ then gives a commuting diagram of solid arrows
\[
\begin{tikzcd}
p_*\mls O_{\mathbf{I}} \arrow[d] \arrow[r,hookleftarrow]& S^\bullet \mls F^{\oplus r} & f^*S^\bullet \mls W \arrow[ll, bend right, "g^{\#}"'] \arrow[d]\arrow[l, dotted] \\
p_*\mls O_{\mathbf{I}} \otimes_{\mls O_\mls X} \mls O_{\mls X}[x_{ij}]_{\det} && f^*S^\bullet \mls W \otimes_{\mls O_\mls X} \mls O_{\mls X}[x_{ij}]_{\det} \arrow[ll]
\end{tikzcd}
\]
and to prove the theorem it suffices to show that when $\mls W$ is polynomial, the morphism $g^{\#}$ factors through $S^\bullet \mls F^{\oplus r}$ as indicated. This may be checked locally on $\mls X$ where we have identifications $\mls F \simeq \mls O^{\oplus r}_{\mls X}$ and $p_*\mls O_{\mathbf{I}} = \mls O_{\mls X}[x_{ij}]_{\det}$; in this case, the left vertical arrow is the diagonal map. When $\mls W$ is polynomial the right vertical arrow has image in $f^*S^\bullet \mls W \otimes_{\mls O_\mls X} \mls O_{\mls X}[x_{ij}]$, and it follows that in this case $g^{\#}$ factors as required.
\end{proof}

\begin{rem}\label{rem:homis}
The stack $\mathbf{A}_{\mls F^{\oplus r}}$ appearing in \ref{P:2.3} is equal to the hom stack
\[
    \mathbf{A}_{\mls F^{\oplus r}} = \Hom_{\mls X}(\mls F, \mls O_{\mls X}^{\oplus r})
\]
and the principal $\mathbf{GL}_r$-bundle $\mathbf{I}$ appearing in the proof is the open substack of isomorphisms
$\mathbf{I} = \underline{\mathrm{Isom}}_{\mls X}(\mls F, \mls O_{\mls X}^{\oplus r}).$
\end{rem}





\section{Ample vector bundles}
\label{S:ample}
Let $R$ be a Noetherian ring and let $\mls X$ be a finite type tame stack over $R$.

\subsection{Sheaves of algebras associated to vector bundles}\label{S:algebras}
 Let $\mls E$ be a vector bundle on $\mls X$. Let $f: \mathbf{E} \to \mc X$ and  $p: \mathbf{P}(\mls E) \to \mls X$ denote the stacks
\[
    \mathbf{E} := \Spec_{\mc X} (S^\bullet \mc E) \quad \quad \text{and} \quad \quad \mathbf{P}(\mls E) := \Proj_{\mls X}(S^\bullet \mls E)
\]
and let $E \to X$ and $h: P \to X$ denote the respective coarse spaces. Then we have
\[
    E = \Spec_X( \oplus \pi_*S^n \mls E) \quad \quad \text{and} \quad\quad P = \Proj_{X}(\oplus \pi_*S^n \mls E).
\]

\begin{lem}\label{L:3.4b} Let $\mls E$ be a vector bundle on $\mls X$, and let $\mls F$ be a coherent sheaf on $\mls X$.

(i) The sheaf of algebras $\oplus _{n\geq 0}\pi _*S^n\mls E$ is finitely generated over $\mls O_X$.

(ii) If $\mls F$ is a coherent sheaf on $\mls X$ then $\oplus _{n\geq 0}\pi _*(\mls F\otimes S^n\mls E)$ is a finitely generated module over $\oplus _{n\geq 0}\pi _*S^n\mls E$.
\end{lem}
\begin{proof}
    It follows from \cite[\href{https://stacks.math.columbia.edu/tag/0DUX}{Tag 0DUX}]{stacks-project} and \cite[\href{https://stacks.math.columbia.edu/tag/0DUZ}{Tag 0DUZ}]{stacks-project} that $E \to X$ is of  finite type which implies (i). Similarly if $\mc F$ is a coherent sheaf on $\mc X$ then $\oplus_{n \geq 0} \pi_*(\mc F \otimes S^n \mc E)$ is the quasi-coherent sheaf of $\oplus_{n \geq 0} \pi_* S^n \mc E$-modules corresponding to the pushforward of $f^*\mc F$ to $E$, and hence is coherent implying (ii).
\end{proof}


\subsection{Criteria for ampleness}\label{S:criteria}
The projective bundle $p:\mathbf{P}(\mc E) \to \mc X$ comes equipped with the universal quotient $p^*\mls E \to \mls O_{\mathbf{P}(\mls E)}(1)$. We fix an integer $N$ such that $\mls O_{\mathbf{P}(\mls E)}(N)$ descends to an invertible sheaf $\mls O_P(1)$ on $P$. Note that $\mls O_P(1)$  is the sheaf called ``the $N^{th}$ twist of the structure sheaf'' in  \cite[\href{https://stacks.math.columbia.edu/tag/01MN}{Tag 01MN}]{stacks-project}, by loc. cit. it is relatively ample for $h$.

\begin{lem}\label{L:1}
 The following are equivalent.
    \begin{enumerate}
        \item [(i)] $\mls E$ is $H$-ample on $\mls X$.
        \item [(ii)] $\mls E$ is faithful and for every coherent sheaf $ F$ on $ X$ there is an integer $n_0$ such that $F \otimes \pi_* S^n \mls E$ is generated by global sections for $n \geq n_0$.
        \item [(iii)] $\mls E$ is faithful and $\mls O_P(1)$ is ample on $P$.
    \end{enumerate}
\end{lem}

\begin{rem}
Even if $\mls E$ is $H$-ample on $\mls X$, the line bundle $\mls O_{\mathbf{P}(\mls E)}(1)$ won't in general be $H$-ample on $\mathbf{P}(\mls E)$ because the stack $\mathbf{P}(\mls E)$ is usually not cyclotomic.
\end{rem}

\begin{rem}
Condition (ii) in \ref{L:1} says that $\mls E$ is faithful and $\oplus_{n \geq 0} \pi_* S^n \mls E$ is an ample sheaf of graded algebras in the sense of \cite{Kubota}. The equivalence of conditions (ii) and (iii) is discussed in \cite{Kubota} when $X$ is a scheme and $\oplus_{n \geq 0} \pi_*S^n\mls E$ is generated in degree 1.
\end{rem}
\begin{proof}[Proof of Lemma \ref{L:1}]
To see that (i) implies (ii), let $F$ be a coherent sheaf on $X$. Since $\pi$ is a coarse moduli morphism we have an identification
\[
 F \otimes \pi_*S^n \mls E = \pi_*(\pi^*F \otimes S^n \mls E)
\]
for all $n$. On the other hand, since $\mls E$ is $H$-ample there is an integer $n_0$ such that for $n\geq n_0$ the right hand side is globally generated.

For (ii) implies (iii), by \ref{T:A.1} it is enough to show that if $G$ is a coherent sheaf on $P$ then there is an integer $n_0$ such that $n \geq n_0$ implies $G(n)$ is globally generated. 
Since $\mc O_P(1)$ is relatively ample, there is an integer $n_1$ such that for $n\geq n_1$ the canonical map
\[
h^*h_*G(n) \to G(n)
\]
is surjective. Setting $F = h_*G(n_1)$, we see that it is enough to show that there is an integer $n_0$ such that $h^*F(n)$ is generated by global sections for $n \geq n_0.$
But for $n$ sufficiently large we have an isomorphism and a surjection
\[
h^*(F \otimes \pi_*S^n \mls E)  \simeq h^*(h_*(h^*F(n))) \twoheadrightarrow h^*F(n)
\]
where the isomorphism follows from \cite[\href{https://stacks.math.columbia.edu/tag/0B5R}{Tag 0B5R}]{stacks-project} (note that to verify that the map is an isomorphism we may work \'etale locally on $X$ where this reference applies) and the surjection is again because $\mc O_P{(1)}$ is relatively ample.  Since $F \otimes \pi_*S^n \mls E$ is globally generated for large enough $n$ by assumption, so is $h^*F(n)$.

Finally, for (iii) implies (i), let $q: \mathbf{P}(\mc E) \to P$ be the coarse moduli map. For a coherent sheaf $\mc F$ on $\mc X$ we have $\mc F \otimes S^n \mc E \simeq p_*(p^*\mc F(n))$  and hence
\begin{equation*}
\pi _*(\mc F\otimes S^n\mc E)\simeq h_*q_*(p^*\mc F(n)).
\end{equation*}
We will show the right hand side is globally generated for large enough $n$. If we write $n = j + mN$ for some $j=0, \ldots, N-1$, then we have
\[
q_*(p^* \mc F(n)) = q_*(p^* \mc F(j)) \otimes \mc O_P(m)
\]
so it is enough to show that there is an integer $m_0$ such that if $G$ is a coherent sheaf on $P$ then $h_*(G(m))$ is globally generated for $m \geq m_0$. This is \ref{P:A.2b}.
\end{proof}

\begin{cor}\label{E:3.6}
    Let $\mc E$ be a faithful vector bundle on $\mc X$ and let $\mc L$ be an ample line bundle on the coarse space $X$.   Then for $m$ sufficiently large the vector bundle $\mc E(m):= \mc E\otimes \pi ^*\mc L^{\otimes m}$ is both $H$-ample and $\det$-ample on $\mls X$.
\end{cor}   
\begin{proof} Since $\mc E$ is faithful so is the sheaf $\mc E(m)$. 
 Consider the projective bundle $p:\mathbf{P}(\mc E)\rightarrow \mc X$ with coarse space $q: \mathbf{P}(\mls E) \to P$.  We have $\mathbf{P}(\mc E(m))\simeq \mathbf{P}(\mc E)$ but the universal bundles are different.  With $N$ and $\mls O_P(1)$ as in \ref{S:criteria} we have
 $$
 \mc O_{\mathbf{P}(\mc E(m))}(N)\simeq \mc O_{\mathbf{P}(\mc E)}(N)\otimes q^*h^*\mc L^{\otimes mN} = q^*(\mc O_P(1) \otimes h^* \mc L^{\otimes mN}).
 $$
 It follows from \cite[\href{https://stacks.math.columbia.edu/tag/0892}{Tag 0892}]{stacks-project} that the right hand side is ample for sufficiently large $m$.   Using \ref{L:1} we conclude that $\mc E(m)$ is $H$-ample for $m$ sufficiently large.  Furthermore, we have $\text{det}(\mls E(m))\simeq \text{det}(\mls E)\otimes \pi ^*\mls L^{\otimes rm}$, where $r$ is the rank of $\mls E$.  It follows that for $m$ sufficiently big some power of $\det(\mls E(m))$ descends to an ample line bundle on $X$ (see for example \cite[II Exercise 7.5 (b)]{MR0463157}); i.e., for such $m$ the vector bundle $\mls E(m)$ is $\det$-ample. 
\end{proof}

Related to this is the following result for line bundles:
\begin{lem}\label{E:cyclotomic} Let $\mls X$ be a cyclotomic stack over $R$ and let $\mls L$ be a line bundle on $\mls X$.  Then $\mls L$ is faithful if and only if $\mls L$ is uniformizing in the sense of \cite[2.3.11]{AH}, and $\mls L$ is $\det $-ample if and only if it is polarizing  in the sense of \cite[2.4.1]{AH}.  Furthermore $\mls L$ is $H$-ample if and only if $\mls L$ is $\det $-ample so that for line bundles the notions of $\det $-ample, $H$-ample, and polarizing are all equivalent.  
\end{lem}
\begin{proof}
The first two statements are immediate from the definitions.  To prove the last statement, 
let $N$ be an integer annihilating all the stabilizer groups of $\mls X$ so that $\mls L^{\otimes N}\simeq \pi ^*\mls M$ for a line bundle $\mls M$ on $X$.  
For any coherent sheaf $\mls G$ on $X$ we have
$$
\pi _*(\pi ^*\mls G\otimes \mls L^{nN})\simeq \mls G\otimes \mls M^n,
$$
so if $\mls L$ is $H$-ample then by \ref{T:A.1} $\mls M$ is ample on $X$ so $\mls L$ is $\det $-ample.

Conversely assume that $\mls M$ is ample on $X$.  Let $\mls F$ be a coherent sheaf on $\mls X$ and let $n$ be an integer.  Write $n = r+Ns$ with $0\leq r<N$ so that $\mls F\otimes \mls L^{\otimes n}\simeq \mls F\otimes \mls L^{\otimes r}\otimes \pi ^*\mls M^{\otimes s}$.  Then we have
$$
\pi _*(\mls F\otimes \mls L^{\otimes n})\simeq (\pi _*(\mls F\otimes \mls L^r))\otimes \mls M^{\otimes s}.
$$
Since $\mls M$ is ample there exists an integer $s_0$ such that these sheaves are generated by global sections for $s\geq s_0$.  Taking $n_0 = s_0N$ we see that $\mls L$ is $H$-ample on $\mls X$.
\end{proof}

\subsection{Properties  of ample bundles}

We show that sums of ample bundles are ample (in either sense, see \ref{L:2}) and that ampleness is preserved by restriction along an immersion (again in either sense, see \ref{lem:det-immersion}).


\begin{lem}\label{L:3.5b} Let $\mls E_1$ and $\mls E_2$ be vector bundles on $\mls X$.  Then $\pi _*(\oplus _{n\geq 0}S^n(\mls E_1\oplus \mls E_2))$ is finitely generated as a module over $(\pi _*(\oplus _{n\geq 0}S^n\mls E_1))\otimes (\pi _*(\oplus _{n\geq 0}S^n\mls E_2)).$
\end{lem}
\begin{proof}
Let $\mc E = \mc E_1 \oplus \mc E_2$ and let $\mathbf{E}, \mathbf{E}_1, \mathbf{E}_2$ and $E, E_1, E_2$ be the associated stacks and coarse spaces as in \ref{S:algebras}. We wish to show that 
\begin{equation}\label{eq:3.5b2}
E \to E_1 \times_X E_2
\end{equation}
is finite. The composition
\begin{equation}\label{eq:3.5b1}
\mathbf{E} = \mathbf{E}_1 \times_\mc X \mathbf{E}_2 \to \mathbf{E}_1 \times_X \mathbf{E}_2 \to E_1 \times_X E_2
\end{equation}
is proper: indeed, the first arrow is the base change of the finite diagonal $\mls X\rightarrow \mls X\times _X\mls X$, and the second arrow is a product of proper morphisms. On the other hand \eqref{eq:3.5b1} also factors as 
$\mathbf{E} \to E \to E_1 \times E_2,$
so by \cite[\href{https://stacks.math.columbia.edu/tag/0CQK}{Tag 0CQK}]{stacks-project} we have that \eqref{eq:3.5b2} is proper. But \eqref{eq:3.5b2} also induces a bijection on points with values in algebraically closed fields, so it is finite. 
\end{proof}

\begin{lem}\label{L:2} 
    Let $\mc E_1$ and $\mc E_2$ be two $H$-ample (resp. det-ample) vector bundles on $\mc X$.  Then $\mc E:= \mc E_1\oplus \mc E_2$ is a $H$-ample (resp. det-ample) vector bundle on $\mc X$.
\end{lem}
\begin{proof}
Since $\text{det}(\mls E) \simeq \text{det}(\mls E_1)\otimes \text{det}(\mls E_2)$, and the tensor product of ample line bundles on $X$ is ample, the statement for $\det $-ample sheaves is immediate.

To prove the statement for $H$-ample line bundles we follow the argument of \cite[Proof of 2.2]{Hartshorneample}.
We have
$$
S^n\mc E\simeq \oplus _{p+q=n}S^p\mc E_1\otimes S^q\mc E_2,
$$
so it suffices to show that for any coherent sheaf $\mc F$ on $\mc X$ there exists an integer $n_0$ such that $\pi _*(\mc F\otimes S^p\mc E_1\otimes S^q\mc E_2)$ is generated by global sections whenever $p+q\geq n_0$. 
Since $\oplus _{p,q\geq 0}\pi _*(\mls F\otimes S^p\mls E_1\otimes S^q\mls E_2)$ is finitely generated over $(\oplus _{p\geq 0}\pi _*S^p\mls E_1)\otimes (\oplus _{q\geq 0}\pi _*S^q\mls E_2)$ by \ref{L:3.5b}, there is a finite set of pairs $\{(p_j, q_j)\}_{j=1}^w$ such that for any $p$ and $q$ the map
\[
\oplus _{j=1}^w\Big(\pi _*(S^{p-p_j}\mls E_1)\otimes \pi _*(S^{q-q_j}\mls E_2)\otimes \pi _*(\mls F\otimes S^{p_j}\mls E_1\otimes S^{q_j}\mls E_2)\Big)\rightarrow \pi _*(\mls F\otimes S^p\mls E_1\otimes S^q\mls E_2)
\]
is surjective. Multiplying the second two factors, we see that this map factors through
\begin{equation}\label{E:generators}
\oplus _{j=1}^w(\pi _*(S^{p-p_j}\mls E_1))\otimes \pi _*(\mls F\otimes S^{p_j}\mls E_1\otimes S^q\mls E_2)\rightarrow \pi _*(\mls F\otimes S^p\mls E_1\otimes S^q\mls E_2)
\end{equation}
and so for all $p$ and $q$ the map \eqref{E:generators} is surjective.

 Now choose integers as follows:
\begin{enumerate}
    \item [(a)] Choose $n_1>0$ such that $\pi _*S^{n-p_j}\mls E_1$ is globally generated for all $n\geq n_1$ and $j=1, \dots, w$.  This is possible because $\mls E_1$ is $H$-ample.
    \item [(b)] Choose $n_2$ such that $\pi _*(\mls F\otimes S^{p_j}\mls E_1\otimes S^n\mls E_2)$ is globally generated for all $n\geq n_2$ and $j=1, \dots, w$.  This is possible because $\mls E_2$ is $H$-ample.
    \item [(c)] For $r=0, \dots, n_1-1$ choose $m_r$ such that for all $n\geq m_r$ the sheaf $\pi _*(\mls F\otimes S^r\mls E_1\otimes S^n\mls E_2)$ is generated by global sections.
    \item [(d)] For $s=0, \dots, n_2-1$ choose $l_s$ such that for all $n\geq l_s$ the sheaf $\pi _*(\mls F\otimes S^n\mls E_1\otimes S^s\mls E_2)$ is generated by global sections for $s\geq l_s$.
\end{enumerate}
Now define $n_0:= \text{max}(r+m_r, s+l_s)$.  We claim that  for $p+q>n_0$ the sheaf $\pi _*(\mls F\otimes S^p\mls E_1\otimes S^q\mls E_2)$ is generated by global sections.  Indeed if $p<n_1$ then $q> m_p$ and we get the global generation by (c).  Similarly if $q<n_2$ then $p>l_s$ and (d) applies.  If both $p\geq n_1$ and $q\geq n_2$ then the source of the surjective map \eqref{E:generators} is globally generated by (a) and (b) so we get the result also in this case.
\end{proof}



\begin{lem}\label{lem:det-immersion}Let $\mls X, \mls X'$ be tame stacks of finite type with finite diagonal over a noetherian ring $R$ and let $i:\mls X' \to \mls X$ be an immersion. If $\mls E$ is an $H$-ample (resp. det-ample) vector bundle on $\mls X$, then $i^{*} \mls E$ is $H$-ample (resp. det-ample) on $\mls X'$.
\end{lem}
\begin{proof}
Since $i$ is representable, the bundle $i^*\mls E$ is faithful. So to show that $i^*$ preserves det-ampleness it is enough to show that restriction from $X$ to $X'$ preserves ample line bundles.
Since $\mls X'$ is separated over $X$ the map on coarse spaces $\bar i:X'\rightarrow X$ is separated and quasi-finite, and therefore quasi-affine \cite[\href{https://stacks.math.columbia.edu/tag/02LR}{Tag 02LR}]{stacks-project}.  From this and \cite[\href{https://stacks.math.columbia.edu/tag/0892}{Tag 0892}]{stacks-project} it follows that the pullback along $\bar i$ of an ample invertible sheaf on $X$ is ample on $X'$. Similarly, to prove the statement for $H$-ampleness it is enough to show that pullback along the canonical morphism $i_P: P_{i^*\mls E} \to P$ preserves ample bundles, where $P_{i^*\mls E}$ is the coarse space of $\mathbf{P}(i^*\mls E)$ and $P$ is the coarase space of $\mathbf{P}(\mls E)$. But $i_P$ is quasi-affine by the same reasoning as above.
\end{proof}

\subsection{Ample vector bundles and generating sheaves}

Recall from \cite[5.1]{OS03} that a vector bundle $\mls E$ on $\mls X$ is \emph{generating} if for all coherent sheaves $\mls F$ on $\mls X$ the evaluation map
$$
(\pi ^*\pi_*\mls Hom(\mls E, \mls F))\otimes \mls E\rightarrow \mls F
$$
is surjective.
By \cite[5.2 and Remark following it]{OS03} a vector bundle $\mls E$ on $\mls X$ is generating if and only if for every geometric point $\bar x\rightarrow \mls X$ every irreducible representation of $G_{\bar x}$ occurs in the representation $\mls E(\bar x)$.
Our goal in this section is to prove \ref{L:3.2}, which states that a vector bundle which is both generating and $H$-ample is det-ample.

\begin{lem}\label{L:3.15} Let $\mls E$ be an $H$-ample vector bundle which is also generating, and let $\mls F$ be a coherent sheaf on $\mls X$.  Then there exists an integer $n_0$ such that for all $n\geq n_0$ there exists an integer $r_n\geq 1$ and a surjection $\xymatrix{\mls E^{\oplus r_n}\ar@{->>}[r]& \mls F\otimes S^n\mls E.}$
\end{lem}
\begin{proof}
    Since $\mls E$ is generating for all integers $n\geq 1$ the evaluation map
    $$
    (\pi ^*\pi _*\mls Hom(\mls E, \mls F\otimes S^n\mls E))\otimes \mls E\rightarrow \mls F\otimes S^n\mls E
    $$
    is surjective.   Now since $\mls E$ is $H$-ample there exists an integer $n_0$ such that for all $n\geq n_0$ the sheaf
    $$
    \pi _*\mls Hom(\mls E, \mls F\otimes S^n\mls E) = \pi _*(\mls E^\vee \otimes \mls F\otimes S^n\mls E)
    $$
is generated by global sections.  Choosing a surjection $\mls O_X^{r_n}\rightarrow \pi _*(\mls E^\vee \otimes \mls F\otimes S^n\mls E)$ for integers $r_n$ we obtain the result.
\end{proof}


\begin{lem}\label{L:3.16} Let $\mls E$ be a $H$-ample vector bundle on $\mls X$ which is also generating.  Then there exists an integer $n_0$ such that for all integers $n\geq n_0$ and $s\geq 1$ there exits an integer $r_{n, s}$ and a surjection
$$
\xymatrix{\mls E^{\oplus r_{n, s}}\ar@{->>}[r]& (S^n\mls E)^{\otimes s}.}
$$
\end{lem}
\begin{proof}
  By \ref{L:3.15} applied to $\mls F = \mls O_{\mls X}$ and $\mls F = \mls E$ there exists an integer $n_0$ such that for all $n\geq n_0$ the sheaves $S^n\mls E$ and $\mls E\otimes S^n\mls E$ are quotients of direct sums of copies of $\mls E$.  Fix an integer $n\geq n_0$.  By induction on $s$ we can find an integer $r_1$ and a surjection $\mls E^{\oplus r_1}\rightarrow (S^n\mls E)^{\otimes (s-1)}$.  Tensoring with $S^n\mls E$ we get a surjection
  $$
  (\mls E\otimes S^n\mls E)^{\oplus r_1}\rightarrow (S^n\mls E)^s.
  $$
Now choosing a surjection $\mls E^{\oplus r_2}\rightarrow \mls E\otimes S^n\mls E$ we get a surjection
$$
(\mls E^{\oplus r_2})^{\oplus r_1}\simeq \mls E^{\oplus r_1r_2}\rightarrow (S^n\mls E)^{\otimes s}.
$$
\end{proof}

\begin{lem}\label{L:3.17} Let $\xymatrix{\mls E\ar@{->>}[r]& \mls E'}$ be a surjective map of vector bundles such that $\mls E$ is $H$-ample.  Then for any coherent sheaf $\mls F$ on $\mls X$ there exists an integer $n_0$ such that for all $n\geq n_0$ the sheaf $\pi _*(\mls F\otimes S^n\mls E')$ is globally generated.
\end{lem}
\begin{proof}
    Indeed  the induced map $\pi _*(\mls F\otimes S^n\mls E)\rightarrow \pi _*(\mls F\otimes S^n\mls E')$ is surjective and global generation of the source implies global generation of the target.
\end{proof}

\begin{lem}\label{L:3.2} Let $\mls E$ be a $H$-ample vector bundle on $\mls X$ which is also generating.  Then $\mls E$ is det-ample.  
\end{lem}
\begin{proof}
    Note that as discussed in the proof of \cite[Proof of 2.6]{Hartshorneample} the determinant of $S^n\mls E$ is a positive power of $\text{det}(\mls E)$. Therefore it is enough to show that there exists an integer $n$ such that if $\text{det}(S^{n}\mls E)^{\otimes N}$ descends to an invertible sheaf $\mls O_X(1)$ for some $N>0$, then $\mls O_X(1)$ is ample.
    
    We take $n=n_0$ to be as in \ref{L:3.16}.  Let $s$ be the rank of $S^{n_0}\mls E$ so that $\text{det}(S^{n_0}\mls E)$ is a quotient of $(S^{n_0}\mls E)^{\otimes s}$.
    If $\text{det}(S^{n_0}\mls E)^{\otimes N}$ descends to a line bundle $\mls O_X(1)$ for some $N>0$, then by \ref{L:3.16} we have surjections
    \[
    \mls E^{\oplus r} \twoheadrightarrow (S^{n_0}\mls E)^{sN} \twoheadrightarrow \text{det}(S^{n_0}\mls E)^{\otimes N} = \pi^* \mls O_X(1)
    \]
    for some $r>0$. Since $\mls E^{\oplus r}$ is $H$-ample by \ref{L:2} we obtain from \ref{L:3.17} that for every coherent sheaf $F$ on $X$ there is an integer $m_0$ such that for all $m \geq m_0$ the sheaf
    \[
    \pi_*(\pi^*F \otimes S^m \pi^* \mls O_X(1)) = F \otimes \mls O_X(n)
    \]
    is globally generated. We conclude that $\mls O_X(1)$ is ample using \ref{T:A.1}.

\end{proof}

\section{Embeddings from ample vector bundles}\label{S:section4}

In this section we  explain how the existence of a $\det$-ample bundle on a proper stack $\mls X$ can be used to embed $\mls X$ into a stack of the form $\QP(V)$ \ref{prop:embed}. We proceed with notation as in \ref{P:1.9}.  So $R$ is a noetherian ring, and if $V$ is a finitely generated $R$-module with $\mathbf{GL}_r$-action, then $\QP (V)$ is the associated tame stack with finite diagonal, which comes equipped with a canonical $\det $-ample vector bundle $\mls E_{\QP (V)}$.

\subsection{Embeddings into $\QP(V)$}\label{S:embed}
Let $\mls X$ be a finite type tame Artin stack with finite diagonal over a noetherian ring $R$ and  let $\mls E$ be a $\det$-ample bundle on $\mls X$ of rank $r$. We now explain how $\mls E$ gives rise to an immersion of $\mls X$ into a stack of the form $\QP(V)$.

\begin{lem}\label{lem:find}
There is a collection of data $(m_0, N, V_1, V_2)$, where $m_0$ and $N$ are positive integers and $V_1$ and $V_2$ are finitely generated $R$-modules with $\mathbf{GL}_r$-action,
satisfying
\begin{enumerate}
\item The algebra $\oplus_{n \geq 0} \pi_*S^n (\mls E^{\oplus r})$ is generated by $\oplus_{0 \leq m \leq m_0} \pi_*S^m (\mls E^{\oplus r})$.
\item The line bundle $(\det(\mls E^{\oplus r}))^{\otimes N}$ descends to a very ample line bundle on $X$, denoted hereafter by $\mls O_X(1)$.
\item $V_1 \subset H^0(X, \mls O_X(1))$ is a finitely generated $R$-submodule and $\mathbf{GL}_r$-submodule inducing an immersion 
\[i:X \hookrightarrow \mathbf{P}(V_1),\] 
where $\mathbf{GL}_r$ acts on $\mls O_X(1)$ via the $N$th power of the determinant character. 
\item $V_2 \subset H^0(X, \oplus_{1 \leq m \leq m_0} \pi_*S^m(\mls E^{\oplus r})\otimes \mls O_X(1))$ is a finitely generated $R$-submodule and $\mathbf{GL}_r$-submodule inducing a surjection
\begin{equation}\label{eq:V2}
V_2 \otimes \mls O_X \to \oplus_{1 \leq m \leq m_0} \pi_*S^m(\mls E^{\oplus r})\otimes \mls O_X(1).
\end{equation}
\end{enumerate}

\end{lem}
\begin{proof}
First note that since $\mls E$ is $\det$-ample, the bundle $\mls E^{\oplus r}$ is also $\det$-ample by \ref{L:2}. We therefore have an ample line bundle $\mls O_X(1)$ on $X$ whose pullback to $\mls X$ is $\det(\mls E^{\oplus r})^{\otimes N}$ for some integer $N$. By \ref{L:3.4b} the algebraic space $\Spec_X(\oplus_{n \geq 0} \pi^*S^n (\mls E^{\oplus r}))$ is locally of finite type over $X$, so there is an integer $m_0$ such that $\oplus_{n \geq 0} \pi_*S^n \mls E^{\oplus r}$ is generated by $\oplus_{0 \leq m \leq m_0} \pi_*S^m \mls E^{\oplus r}$. We can therefore find an integer $n_0$ such that the following hold:
\begin{itemize}
\item The sheaf $\mls O_X(n)$ is very ample on $X$  for $n \geq n_0$.
\item The sheaf $\oplus_{1\leq m \leq m_0} \pi_*S^m \mls E ^{\oplus r}\otimes \mls O_X(n)$ is generated by global sections for $n \geq n_0$. 
\end{itemize}
By replacing $\mls O_X(1)$ with a positive multiple of itself we may take $n_0=1$. 

At this point we have found $m_0$ and $N$ as required in the lemma; the existence of $V_1$ and $V_2$ as $R$-modules follows from \cite[\href{https://stacks.math.columbia.edu/tag/01VR}{Tag 01VR}]{stacks-project}. Since $\mathbf{GL}_r$ acts on $\mls O_X(1)$ through a character $V_1$ is also a $\mathbf{GL}_r$-module, and by possibly enlarging $V_2$ we can also arrange that $V_2$ is $\mathbf{GL}_r$-invariant.  Indeed if $M$ is a $\mathbf{GL}_{r, R}$-representation over $R$ and $V\subset M$ is a finitely generated $R$-submodule then $V$ is contained in a finitely generated $\mathbf{GL}_r$-invariant submodule.  This is a classical fact in representation theory but can also be deduced from \cite[15.4]{MR1771927}.  Indeed using the dictionary between $\mathbf{GL}_{r, R}$-representations over $R$ and quasi-coherent sheaves on $B\mathbf{GL}_{r, R}$ we see that $M$ can be written as a filtered colimit of finitely generated subrepresentations, and since $V_2$ is finitely generated it must be contained in one of them. (Note that if $\mls X$ is proper then we can just take $V_1$ and $V_2$ to be the entire spaces of global sections).
\end{proof}

Let $H$ be the coarse moduli space of $\mathbf{H}:=\mathbf{A}_{\mls E^{\oplus r}}$. Recall from \ref{rem:homis} that $\mathbf{H} = \Hom_{\mls X}(\mls E, \mls O_{\mls X}^{\oplus r})$, so we have a morphism
\begin{equation}\label{eq:det2}
{\det}^N: \mathbf{H} \to \mathbf{A}_{\mls O_X(1)}.
\end{equation}
sending a map $\gamma: \mls E \to \mls O^{\oplus r}$ to $\det(\gamma)^{\otimes N}: \mls O_X(1) \simeq \det(\mls E)^{\otimes N} \to \mls O_X$. Let $\mathbf{I} \subset \mathbf{H}$ be the open substack where \eqref{eq:det2} is nonzero, and note that this agrees with the substack of the same name in the proof of \ref{P:2.3}. As explained there the projection $\mathbf{I} \to \mls X$ is a principal $\mathbf{GL}_r$-bundle.
Setting $V = V_1 \oplus V_2$ we have a sequence of $\mathbf{GL}_r$-equivariant morphisms
\begin{equation}\label{eq:immersion}
\mathbf{I} \xrightarrow{(f, \,\det^N)}H \times_X \mathbf{A}_{\mls O_X(1)} \xrightarrow{g} \mathbf{A}_{V}
\end{equation}
defined as follows. The morphism $f: \mathbf{I} \to H$ is the inclusion $\mathbf{I} \hookrightarrow \mathbf{H}$ followed by the coarse moduli map. The scheme $H \times_X \mathbf{A}_{\mls O_X(1)}$ is given explicitly by
\[
    H\times _X \mathbf{A}_{\mls O_X(1)} = \underline {\Sp }_X(S^\bullet (\mc E^{\oplus r}))\times _X\underline {\Sp }_X(S^\bullet(\mc O_X(1))) = \underline {\Sp }_X(\oplus _{n, m\geq 0}S^m(\mc E^{\oplus r})(n))
    \]
    so the arrow $g$ is induced by the subspaces $V_1$ and $V_2$.

\begin{rem}\label{R:4.3} The map $H\rightarrow \mathbf{A}_{\mls O_X(1)}$ defined by $\det ^N$ corresponds to a map $\delta :\mls O_X(1)\rightarrow \oplus _{n\geq 0}\pi _*S^n(\mls E^{\oplus r})$ over $X$. This map is equivariant with respect to the $\mathbf{GL}_r$-action.  In particular, composing the inclusion $V_2\subset H^0(X, \oplus _{n\geq 0}\pi _*S^n(\mls E^{\oplus r})(1))$ with the map 
$$
H^0(X, \oplus _{n\geq 0}\pi _*S^n(\mls E^{\oplus r})(1))\rightarrow H^0(X, \oplus _{n\geq 0}\pi _*S^n(\mls E^{\oplus r}))
$$
defined by $\delta $ we get a $\mathbf{GL}_r$-equivariant map $V_2\rightarrow H^0(X, \oplus _{n\geq 0}\pi _*S^n(\mls E^{\oplus r}))$ which defines a map $\mathbf{H}\rightarrow \mathbf{A}_{V_2}$ (as in \ref{P:2.3}, but see also \ref{R:2.8}) giving rise to the map obtained from \eqref{eq:immersion} by restriction to $\mathbf{I}$.
\end{rem}

\begin{thm}\label{prop:embed}
The composition \eqref{eq:immersion} is a immersion
factoring through $\mathbf{A}^{s,\; \det}_{V}$. Taking the quotient by $\mathbf{GL}_r$ we obtain an immersion $\mls X \to \QP(V)$.
\end{thm}
\begin{proof}
Set $L:= \mathbf{A}_{\mls O(1)}$. Consider a section $s\in V_1$ and the nonvanishing locus $D(s)\subset \mathbf{A}_{V_1}$ with corresponding open subset $D_+(s)\subset \mathbf{P}({V_1})$.  Let $X_s\subset X$ denote the open set $X\cap D_+(s)$ and let $L_s$ denote the restriction of $L$ to $X_s$ so we have a cartesian square
\begin{equation}\label{eq:immerse1}
\xymatrix{
L^*_s\ar[r]\ar[d]& D(s)\ar[d]\\
X_s\ar@{^{(}->}[r]^-i& D_+(s),}
\end{equation}
where the vertical morphisms are $\mathbf{G}_m$-torsors and $L^*\subset L$ is the complement of the zero section.  
Note also that $\mls O_X(1)$ is trivialized over $X_s$ by the section $s$ and therefore 
$L^*_s \simeq \mathbf{G}_{m, X_s}$

Let $H_s$ (resp. $\mathbf{I}_s$) denote the restriction of $H$ (resp. $\mathbf{I}$) to $X_s$.
Then \eqref{eq:immerse1} can be extended to a commuting diagram
\begin{equation}\label{eq:immerse2}
\begin{tikzcd}
& H_s \times_{X_s} L_s^* \arrow[r, "j"] \arrow[d] & \mathbf{A}_{V_2} \times L^*_s \arrow[d] \arrow[r] & \mathbf{A}_{V_2} \times D(s) \arrow[d] \\
\mathbf{I}_s \arrow[ur, "{(f,\, \det^N)}"]\arrow[r, "f"]& H_s \arrow[r]& X_s \arrow[r, hookrightarrow, "i"] & D_+(s) 
\end{tikzcd}
\end{equation}
where the rightmost square is fibered and the composition $H_s \times_{X_s} L_s^* \to \mathbf{A}_{V_2} \times D(s)$ is a restriction of $g$.

\begin{lem}\label{lem:embed}
The morphisms $j$ and $(f,\,\det^N)$ in \eqref{eq:immerse2} are closed immersions.

\end{lem}
\begin{proof}
Both sides of the morphism $j$ are affine over $X_s$ and $j$ is induced by a map of quasi-coherent sheaves of algebras
\begin{equation}\label{eq:immerse4}
\mls O_{X_s}[t^\pm ]\otimes _RS^\bullet V_2\rightarrow \mls O_{X_s}[t^\pm ] \otimes _{\mls O_{X_s}}(\oplus_{m \geq 0} \pi_*(S^m \mls E^{\oplus r})|_{X_s}).
\end{equation}
To describe the algebra map, note that we have an identification
\[
\mls O_{X_s}[t^\pm ] \otimes _{\mls O_{X_s}}(\oplus_{m \geq 0} \pi_*(S^m \mls E^{\oplus r})|_{X_s}) \simeq \oplus_{{n \in \mathbf{Z},\, m \geq 0}}\left( \mls O_{X_s}(n) \otimes_{\mls O_{X_s}} \pi_*(S^m \mls E^{\oplus r})|_{X_s})\right)
\]
sending $t$ to the section $s \in \mls O_{X_s}(1)$. Now \eqref{eq:immerse4} is induced by sending $v \in V_2$ to its natural image in the $n=1$ summand of the right hand side. By assumption (1) and the fact that $t$ is a unit, the right side is generated as an $\mls O_{X_s}[t^\pm ]$-algebra by $\oplus _{1\leq m\leq m_0}\pi _*S^m(\mls E^{\oplus r})(1)$, and therefore using (4) it is generated as an $\mls O_{X_s}[t^\pm ]$-algebra by the image of $V_2$ from which it follows that $j$ is a closed immersion.


We next note that $f: \mathbf{I} \to H$ is an open immersion. Indeed, $\mathbf{I}$ is an algebraic space since $\mls E$ is faithful, so its coarse moduli map is an isomorphism, and coarse moduli morphisms are compatible with pullback along open immersions. It follows that the square
\[
\begin{tikzcd}[column sep=.75in]
\mathbf{I}_s \arrow[r, "{(f,\,\det^N)}"] \arrow[d, "f"] & H_s \times_{X_s} L_s^* \arrow[d]\\
H_s \arrow[r, "{(id, \,\det^N)}"] & H_s \times_{X_s} L_s
\end{tikzcd}
\]
is fibered. Since the bottom arrow is a closed immersion, being a section of a separated morphism, the arrow $(f,\,\det^N)$ is also a closed immersion.

\end{proof}

Continuing with the proof of \ref{prop:embed}, it follows from the diagram \eqref{eq:immerse2} and \ref{lem:embed} that the restriction of \eqref{eq:immersion} to $\mathbf{I}_s$ is an immersion (recall that $i$ is an immersion by assumption (3) in \ref{lem:find}). Hence \eqref{eq:immersion} is an immersion.

Let $\bar x \to I$ be a geometric point and let $s\in V_1$ be a section which does not vanish at $\bar x$ (such a section exists because $X$ embeds into $\mathbf{P}(V_1)$). Let $D(s) \times \mathbf{A}_{V_2}$ be the corresponding basic open in $\mathbf{A}_V$ and let 
\[
(D(s) \times \mathbf{A}_{V_2})_{\kappa(\bar x)}
\]
denote the fiber over $\Sp(\kappa(\bar x)) \to R$. 

We claim the orbit of $\bar x$ in $(D(s) \times \mathbf{A}_{V_2})_{\kappa(\bar x)}$ is closed. For this let $\bar y  \to X_s $ be the image of $\bar x$ (we also write $\bar y \to \mathbf{P}(V_1)$ for the composition with the immersion $i$). Since $\mathbf{P}(V_1)$ is separated over $R$, it is enough to show that the orbit of $\bar x$ is closed in $(D(s) \times \mathbf{A}_{V_2})_{\bar y}$, but this is equal to the fiber $(\mathbf{A}_{V_2} \times L_s^*)_{\bar y}$ by the diagram \eqref{eq:immerse2}. By \ref{lem:embed} it is therefore enough to show that the orbit of $\bar x$ in $\mathbf{I}_{\bar y}$ is closed. But since $\mathbf{I} \to \mls X$ is a $\mathbf{GL}_r$-torsor the orbit of $\bar x$ is equal to $\mathbf{I}_{\bar y}$.

 Finally, since the diagonal of $\mls X$ is finite the stabilizer group scheme of $\bar x$ is also finite, which shows $\bar x \rightarrow \mathbf{A}_V$ is stable.
\end{proof}

We can now prove \ref{cor:4.6} from the introduction.

\begin{proof}{Proof of \ref{cor:4.6}}
Assume (1). Then $\mls X$ admits a faithful vector bundle \cite[5.5]{OS03}.  Tensoring with a sufficiently high power of the pullback of an ample line bundle on $X$ yields a vector bundle on $\mls X$ that is both $H$-ample and $\det$-ample by \ref{E:3.6}. So (2) holds.

It is clear that (2) implies (3), and (3) imples (4) by \ref{prop:embed}. Finally assume (4). Then $\mls X$ admits a $\det$-ample vector bundle by \ref{lem:det-immersion}, and in particular its coarse space is quasi-projective. We get a global quotient presentation for $\mls X$ by restricting the one for $\QP(V)$.
\end{proof}

\begin{example}\label{ex:P13} Consider the weighted projective stack $\mls P(1,3)=[\mathbf{A}^2_{x, y}-\{0\}/\mathbf{G}_m]$ over a ring $R$, where $\mathbf{G}_m$ acts on $x$ with weight $1$ and on $y$ with weight $3$.  The coarse space is $\mathbf{P}^1=[\mathbf{A}^2_{w, y}-\{0\}/\mathbf{G}_m]$ with the map $\pi :\mls P(1, 3)\rightarrow \mathbf{P}^1$ induced by the map $R[w, y]\rightarrow R[x, y]$, $w\mapsto x^3$ and $y\mapsto y$, which is equivariant with respect to the morphism $\mathbf{G}_m\rightarrow \mathbf{G}_m$ sending $u$ to $u^3$. 

The stack $\mls P(1,3)$ is the third root stack associated to the divisor $p = [0:1]\in \mathbf{P}^1$, and we define $\mls E$ to be the line bundle on $\mls P(1, 3)$ satisfying
\begin{equation}\label{eq:P13.1}
\mls E^{\otimes 3} = \pi^* \mls O_{\mathbf{P}^1}(p).
\end{equation}
We have an identification of graded $R$-algebras
\begin{equation}\label{eq:P13.2}\oplus _{n\geq 0}H^0(\mathbf{P}^1, \pi_*S^n\mls E) = \oplus _{n\geq 0}H^0(\mathbf{P}^1, \pi_*\mls E^{\otimes n}) \simeq  R[x, y].\end{equation}


We compute an embedding of $\mls P(1, 3)$ determined by $\mls  E$ as in \ref{prop:embed}. From \eqref{eq:P13.1} we can take $N=3$ and from \eqref{eq:P13.2} we can take $m_0=3$ in \ref{lem:find}. Again using \eqref{eq:P13.2} we find
\[
V_1 = H^0(\mathbf{P}^1, \mls O_{\mathbf{P}^1}(1)) = H^0(\mathbf{P}^1, \pi_*\mls E^{\otimes 3}) = Rx^3 \oplus Ry
\]
and similarly
\begin{align*}
V_2 &= \oplus_{m=1, 2, 3} H^0(\mathbf{P}^1, (\pi_*\mls E^{\otimes m}) \otimes \mls O_{\mathbf{P}^1}(1)) = \oplus_{m=4, 5, 6} H^0(\mathbf{P}^1, \pi_*\mls E^{\otimes m})\\
&= (Rx^4 \oplus Rxy) \oplus (Rx^5 \oplus Rx^2y) \oplus (Rx^6 \oplus Rx^3y \oplus Ry^2).
\end{align*}
Hence the immersion of \ref{prop:embed} is the morphism
$$
\mls P(1,3)\hookrightarrow \mls P(3,3,4,4,5,5,6,6,6)
$$
given in coordinates by
$$
[a:b]\mapsto [a^3:b:a^4:ab: a^5:a^2b: a^6:a^3b: b^2].
$$

\end{example}

\begin{rem}
Compare Example \ref{ex:P13} with the construction in \cite[Cor 2.4.4]{AH} which uses $\mls E$ to embed $\mls P(1, 3)$ into the weighted projective stack $\mls P(1, 2, 3, 3).$ 
\end{rem}

\begin{rem}
In Example \ref{ex:P13}, the stack $\mls P(1, 3)$ is equal to $\QP(V)$ for some $\Gm$-module $V$ and the bundle $\mls E$ is the canonical $\det$-ample bundle $\mls E_{\QP(V)}$. We see that in the embedding induced by the canonical $\det$-ample bundle on a stack $\QP(V)$ is not the identity in general.
\end{rem}

\subsection{Recovering the data defining the immersion}\label{S:recover}
In this section we explain how from the immersion $\mls X \to \QP(V)$ we can recover $\mls E$ as well as the tuple $(m_0, N, V_1, V_2)$ used to construct it. In particular we recover the embedding $X \to \mathbf{P}(V_1)$. 

To begin, note that there is a rational map $\QP(V) \dashrightarrow \mathbf{P}(V_1)$ induced by the projection $\QP(V_1 \oplus V_2) \dashrightarrow  [\mathbf{A}_{V_1}-\setminus \{0\}/\mathbf{GL}_r]$ followed by the morphism
\[
[(\mathbf{A}_{V_1} \setminus \{0\}) / \mathbf{GL}_r] \to [(\mathbf{A}_{V_1} \setminus \{0\}) / \Gm]
\]
induced by the $N^{th}$ power of the determinant character $\det^N: \mathbf{GL}_r \to \Gm$. The pullback of $\mls O_{\mathbf{P}(V_1)}(1)$ under this rational map is thus canonically identified with the restriction of $\det(\mls E_{\QP(V)})^{\otimes N}$.

Now we recover the data $(\mls E, m_0, N, V_1, V_2)$ from the immersion.

First, we let $\mls E$ be the restriction of $\mls E_{\QP(V)}$. This recovers the original ample bundle because the immersion $\mls X \to \QP(V)$ is the quotient of the $\mathbf{GL}_r$-equivariant morphism \eqref{eq:immersion}.

Next, from the universal property of the coarse moduli map $\mls X \to X$ we have a unique morphism $X \to \mathbf{P}(V_1)$ fitting in a commuting diagram
\begin{equation}\label{eq:p17.1}
\begin{tikzcd}
\mls X \arrow[r, hookrightarrow] \arrow[d, "\pi"] & \QP(V_1 \oplus V_2) \arrow[d, dashrightarrow] \\
X \arrow[r, "i'"] & \mathbf{P}(V_1).
\end{tikzcd}
\end{equation}
The given morphism $i: X \to \mathbf{P}(V_1)$ also fits into such a diagram, so we have $i=i'$. Moreover both $i'^*\mls O_{\mathbf{P}(V_1)}(1)$ and $\mls O_X(1)$ pull back to $\det(\mls E)^{\otimes  N}$ on $\mls X$, so since $\pi^*$ on quasi-coherent sheaves is fully faithful  we have an identification $i'^*\mls O_{\mathbf{P}(V_1)}(1) \simeq \mls O_X(1)$. In this way we recover $V_1 \subset H^0(X, \mls O_X(1))$.

Finally from the universal property of morphisms to $\QP(V)$, we see that the immersion $\mls X \to \QP(V)$ induces a map
\[
V_2 \otimes \mls O_X \to \oplus_{m \geq 0} \pi_*S^m(\mls E^{\oplus r})
\]
(see \ref{eq:2.3-ring}). Multiplying by the section $\det^{-N}$ of $H^0(I, \mls O_X(1)|_{I})$ gives a map 
\begin{equation}\label{eq:p17.2}
V_2 \to H^0(X, \oplus_{m \geq 0} \pi_*S^m(\mls E^{\oplus r})(1))
\end{equation}
which by the discussion in \ref{R:4.3}  equals the starting inclusion of $V_2$ in the right hand side. 



\section{Relatively ample vector bundles}\label{S:section6}

We now generalize the discussion in sections 3 and 4 by replacing the base ring $R$ with a scheme $S$ (which we will require to be quasi-compact, or locally noetherian, or both).

\subsection{Fiberwise criterion for ampleness}
We first show that ampleness of a vector bundle on $\mls X$ can be checked on fibers.

\begin{prop}\label{P:6.3} Let $S$ be a locally noetherian  scheme, let $f:\mls X\rightarrow S$ be a 
proper tame Artin stack over $S$, and let $\mls E$ be a vector bundle on $\mls X$.  If for some point $s\in S$ the restriction $\mls E_s$ of $\mls E$ to the fiber $\mls X_s$ of $\mls X$ over $s$ is det-ample (resp. $H$-ample) then there exists an affine open neighborhood $s\in U\subset S$ containing $s$ such that the restriction of $\mls E$ to $f_U:\mls X_U := f^{-1}(U)\rightarrow U$ is det-ample (resp. $H$-ample).  
\end{prop}
\begin{proof}
    After shrinking on $S$ around $s$ we may assume that $S$ is quasi-compact and in that case we can choose an integer $N$ such that $\det (\mls E)^{\otimes N}$ descends to an invertible sheaf $\mls O_X(1)$ on $X$. The next lemma says that after further shrinking on $S$ around $s$ we may assume that $\mls E$ is faithful.   

\begin{lem}\label{L:5.3b} If $\mls E_s$ is faithful then there exists an open neighborhood of $s$ in $S$ over which the vector bundle $\mls E$ on $\mls X$ is faithful.
\end{lem}
\begin{proof}
    Let $I_{\mls X}\rightarrow \mls X$ denote the inertia stack (which is finite over $\mls X$ since $\mls X$ has finite diagonal).  Consider the natural homomorphism of relative group schemes  $\rho :I_{\mls X}\rightarrow \mathbf{GL}(\mls E)$ over $\mls X$.  Since both source and target is affine over $\mls X$ this map $\rho $ is given by a map of quasi-coherent sheaves of algebras 
    \begin{equation}\label{eq:p18}
    S^\bullet (\mls E\otimes \mls E^\vee )_{\text{det}}\rightarrow \mls A,
    \end{equation}
    where on the left we consider the localization at the determinant of the symmetric algebra on $\mls E\otimes \mls E^\vee $ and the target $\mls A$ is coherent on $\mls X$ since $I_{\mls X}$ is finite over $\mls X$.   By Nakayama's lemma the locus in $\mls X$ where this map is surjective is open.  Let $\mls Z\subset \mls X$ be the complement.  Since $f$ is proper the image $f(\mls Z)\subset S$ is closed and does not meet $s$.  Replacing $S$ by the complement of $f(\mls Z)$ we can then arrange that $\mls Z = \emptyset $, which implies that $\mls E$ is faithful. 
\end{proof}

To prove the Lemma for $H$-ampleness, 
 consider the projective bundle $\mathbf{P}(\mls E)$ and its associated coarse moduli space $P\rightarrow X$.  Let $N>0$ be an integer such that $\mls O_{\mathbf{P}(\mls E)}(N)$ descends to an invertible sheaf $\mls O_P(1)$ on $P$.  If $\mls E_s$ is $H$-ample then by \ref{L:1} the sheaf $\mls O_P(1)$ restricts to an ample invertible sheaf on the fiber $P_s$. By \cite[\href{https://stacks.math.columbia.edu/tag/0D3A}{Tag 0D3A}]{stacks-project} we have that $\mls O_P(1)$ is ample on $P$ over some neighborhood of $s$ in $S$. So $\mls E$ is $H$-ample on this neighborhood by \ref{L:1}.

Similarly, if $\mls E_s$ is $\det$-ample then $\mls O_X(1)$ restricts to an ample invertible sheaf on $X_s$. Again using \cite[\href{https://stacks.math.columbia.edu/tag/0D3A}{Tag 0D3A}]{stacks-project} we have that $\mls O_X(1)$ is ample on $X$ over some neighborhood of $s$ in $S$. So $\mls E$ is $\det$-ample on this neighborhood.

\end{proof}

\subsection{Relatively ample sheaves over a noetherian base}

Following the definition for schemes \cite[\href{https://stacks.math.columbia.edu/tag/01VH}{Tag 01VH}]{stacks-project}, we define the relative notion of an ample vector bundle on a stack. 

\begin{defn} Let $S$ be a noetherian scheme, let $f:\mls X\rightarrow \mls Y$ be a morphism of tame finite type $S$-stacks, and let $\mls E$ be a vector bundle on $\mls X$.  We say that $\mls E$ is \emph{relatively det-ample} (resp. \emph{relatively H-ample}) 
if for every affine scheme $U$ of finite type over $S$ and $S$-morphism $u:U\rightarrow \mls Y$ the pullback $\mls E_U$ of $\mls E$ to $\mls X_U:= \mls X\times _{\mls Y}U$ is $\det$-ample (resp. $H$-ample). \end{defn}

\begin{rem} One might also use the term \emph{$\det$-ample on $\mls X/\mls Y$} (resp. \emph{H-ample on $\mls X/\mls Y$}) as a synonym for relatively $\det$-ample (resp. relatively $H$-ample).
\end{rem}

\begin{rem}
Since $\mls X$ is quasi-compact, some power of $\det(\mls E)$ will descend to a line bundle $\mls O_X(1)$ on $X$. A faithful vector bundle $\mls E$ is relatively $\det$-ample on $\mls X/S$ if and only if the line bundle $\mls O_X(1)$ is relatively ample on $X/S$.
\end{rem}



\begin{prop}\label{prop:fibers} Let $f:\mls X\rightarrow S$ be a proper tame stack over a noetherian scheme $S$ and let $\mls E$ be a vector bundle on $\mls X$.  The following are equivalent:
\begin{enumerate}
    \item [(i)] $\mls E$ is relatively det-ample (resp. $H$-ample).
    \item [(ii)] For every geometric point $\bar s\rightarrow S$ the restriction $\mls E_{\bar s}$ of $\mls E$ to $\mls X_{\bar s}:= \mls X\times _S\bar s$ is det-ample (resp. $H$-ample).
    \item [(iii)] For every point $s\in S$ the restriction $\mls E_s$ of $\mls E$ to the fiber $\mls X\times _S\Sp (\kappa (s))$ is det-ample (resp. $H$-ample).
\end{enumerate}
\end{prop}

\begin{proof}
First note that $\mls E$ is faithful if and only if its restriction $\mls E_{\bar s}$ to every geometric fiber is faithful, if and only if its restriction $\mls E_s$ to every fiber is faithful. This follows from \ref{L:5.3b} and the observation that $\mls E_s$ on $\mls  X_s$ (resp. $\mls E_{\bar s}$ on $\mls X_{\bar s}$) is faithful if and only if the analog of \eqref{eq:p18} over $\mls X_s$ (resp. $\mls X$) is surjective. 

We next recall that if $\mls E$ is a faithful bundle, it is det-ample if and only if $\mls O_X(1)$ is ample on $X$ and by \ref{L:1} it is $H$-ample if and only if $\mls O_P(1)$ is ample on $P$, the coarse space of the projective bundle associated to $\mls E$. Hence the equivalence of (i) and (iii) follows from \ref{P:6.3} and \cite[\href{https://stacks.math.columbia.edu/tag/0893}{Tag 0893}]{stacks-project}, and (ii) is equivalent to (iii) because a line bundle on a scheme $Y$ over a field $k$ is ample if and only if its base change to $Y_{\bar k}$ is ample, where $\bar k$ is the algebraic closure (this follows for example from descent of polarized schemes \cite[4.4.10]{Olssonbook}). 




\end{proof}

\subsection{Embeddings over  non-noetherian base schemes}
We now generalize further and fix $S$ a scheme that is quasi-compact but not necessarily noetherian. This added generality will be used in our discussion of moduli in the following sections.


 Let $\mls V$ be a finitely presented quasi-coherent sheaf on $S$ with a polynomial left $\mathbf{GL}_r$-action. 
 Recall (see \ref{R:relative-actions}) that this is equivalent to a right action of $\mathbf{GL}_r$ on $\mathbf{A}_{\mls V}$. Since $\mls V$ is finitely presented, the scheme $\mathbf{A}_{\mls V}$ is locally of finite presentation over $S$. 
 We let $\mathbf{A}_{\mls V}^{s, \;\det}$ be the stable locus as in \ref{T:GIT}. It follows from  \ref{T:GIT} as well as
\cite[\href{https://stacks.math.columbia.edu/tag/01TT}{Tag 01TT}]{stacks-project} and \cite[\href{https://stacks.math.columbia.edu/tag/06Q9}{Tag 06Q9}]{stacks-project} that the quotient $[\mathbf{A}^{s, \,\det}_{\mls V}/\mathbf{GL}_r]$ is locally of finite presentation over $S$ with finite diagonal, hence by
\cite[Proposition 3.6]{tame} contains a maximal open tame substack which we denote $\QP(\mls V)$. As in the case when $S$ is an affine scheme, $\QP(\mls V)$ has a canonical faithful rank-$r$ vector bundle $\mls E_{\QP(\mls V)}$, and in the case when $S$ is noetherian the vector bundle $\mls E_{\QP(\mls V)}$ is even relatively $\det$-ample as this can be verified Zariski locally on $S$.




\begin{prop}\label{P:non-noetherian}
Let $f:\mls X\rightarrow S$ be a proper flat tame stack over a quasi-compact scheme $S$. The following are equivalent:
\begin{enumerate}
\item [(i)] $\mls X$ admits an immersion over $S$ into a stack of the form $\QP(\mls V)$.
\item [(ii)] $\mls X$ admits a vector bundle which is det-ample in every geometric fiber.
\end{enumerate}
If moreover $S$ is noetherian, then (i)-(ii) are equivalent to 
\begin{enumerate}
\item [(iii)] $\mls X$ admits a relatively det-ample vector bundle.
\end{enumerate}
\end{prop}
\begin{proof}
To see that (i) implies (ii) note that  if $\bar s \to S$ is a geometric point the fiber $\mls X_{\bar s}$ admits an immersion into the stack $\QP(\mls V_{\bar s})$ (note that formation of $\QP(\mls V)$ commutes with restriction of $\mls V$ since the stable locus commutes with basechange \ref{S:GIT without noeth}). Then $\mls E_{\QP(\mls V_{\bar s})}$ restricts to a $\det$-ample on $\mls X/S$ by \ref{lem:det-immersion}. 

To see that (ii) implies (i), 
 let $\mls E$ be a vector bundle on $\mls X$ whose restriction to every geometric fiber of $\mls X \to S$ is $\det$-ample, and let $\bar f: X \to S$ be the morphism induced by $f$. 

 \begin{lem} There exists an integer $m_0$ such that 
 $\oplus_{n \geq 0} S^n(\mls E^{\oplus r})$ is generated as an $\mls O_X$-algebra by $\oplus_{1 \leq m \leq m_0} \pi_*S^m(\mls E^{\oplus r})$.
 \end{lem}
 \begin{proof}
     Since $S$ is quasi-compact it suffices to prove the lemma after replacing $S$ by a Zariski cover, so we may assume that $S$ is affine, say $S = \Sp (R)$.  Then using the finite presentation assumption we can find a noetherian subring $R_0\subset R$ such that $(\mls X, \mls E)$ is obtained by base change from a pair $(\mls X_0, \mls E_0)$ consisting of a proper tame stack $\mls X_0/R_0$ and a vector bundle $\mls E_0$ over $\mls X_0$ which is $\det $-ample in every geometric fiber.  This reduces the proof to the noetherian case where it follows from \ref{lem:find}.
 \end{proof}

Since $S$ is quasi-compact we can find an integer $N$ such that $\det(\mls E)^{\otimes N}$ descends to a line bundle $\mls O_X(1)$ on $X.$ The assumption that $\mls E$ is det-ample in every fiber implies that $\mls O_X(1)$ is ample in every geometric fiber of $X \to S$, hence relatively ample by \cite[\href{https://stacks.math.columbia.edu/tag/0D3D}{Tag 0D3D}]{stacks-project}. By \cite[\href{https://stacks.math.columbia.edu/tag/01VU}{Tag 01VU}]{stacks-project} and quasi-compactness of $S$ we can further choose $N$ so that 
\begin{itemize}
\item [(a)] $\mls O_X(1)$ is relatively very ample for $X\to S$.  
\item [(b)] $\pi_*S^m(\mls E^{\oplus r}\otimes \mls O_X(1))$ is generated by global sections for $1 \leq m \leq m_0$. 
\item [(c)] $H^i(X_{\bar s}, \pi_*S^m(\mls E_{\bar s}^{\oplus r}\otimes \mls O_{X_{\bar s}}(1)) = 0$ for $i>0$ and $0 \leq m \leq m_0$ and all geometric points $\bar s\rightarrow S$.
\end{itemize}

Set $\mls V_1 = \bar f_* \mls O_X(1)$ and $\mls V_2 = \oplus_{1 \leq m \leq m_0} \bar f_*(\pi_* S^m(\mls E^{\oplus r})(1))$. Note that $\mls V_1$ and $\mls V_2$ are locally free sheaves whose formation commute with arbitrary base change: This can be verified locally so it suffices to consider the case when $S$ is affine where by our finite presentation assumptions $(\mls X, \mls E)$ is obtained by base change from a noetherian ring over which (a)-(c) also hold, which implies the assertion by cohomology and base change. 
 Thus if we set $\mls V = \mls V_1 \oplus \mls V_2$ we may form $\QP(\mls V)$, and we obtain a morphism $\mls X \to [\mathbf{A}_{\mls V}/\mathbf{GL}_r]$ via a relative version of the discussion in \ref{S:embed}.  To see that $\mls X \to [\mathbf{A}_{\mls V}/\mathbf{GL}_r]$ is an immersion factoring through $\QP(\mls V)$ we may work Zariski-locally on $S$ so it suffices to consider the case when $S = \Sp(R)$ is affine. Writing $R$ as a colimit of finite type $\mathbf{Z}$-algebras and using that the stacks $[\mathbf{A}_{\mls V}/\mathbf{GL}_r]$, $\QP(\mls V)$ and $\mls X$ are of finite presentation over $S$ we reduce to the case when $R$ is of finite type over $\mathbf{Z}$, where the result follows from \ref{prop:embed}.

The equivalence of (i)-(ii) with (iii) when $S$ is noetherian follows from \ref{prop:fibers}.

\end{proof}

\begin{rem}\label{R:non-noetherian}
Let $\mls X/S$ be as in \ref{P:non-noetherian} and let $\mls E$ be a vector bundle on $\mls X$. Then there is an open subscheme $S' \subset S$ with the property that a geometric point $\bar s \to S$ factors through $S'$ if and only if the fiber $\mls E_{\bar s}$ is det-ample on $\mls X_{\bar s}$. To construct it, we may assume $S$ is affine, and then since $\mls X$ and $\mls E$ are of finite presentation we can find a morphism from $S$ to a noetherian scheme $T$ and a pair $(\mls X_T, \mls E_T)$ over $T$ restricting to $(\mls X, \mls E)$ over $S$. We can construct the desired open subset $T'$ of $T$ using \ref{P:6.3} and \ref{P:non-noetherian}, and then take $S' = S \times_T T'$.
\end{rem}

\begin{example} Let $\mls Y$ be an algebraic stack and let $\mls A_\bullet = \oplus _{n\geq 0}\mls A_n$ be a quasi-coherent locally finitely generated sheaf of $\mls O_{\mls Y}$-algebras, and consider the associated morphism
$$
\mls X:= \mls Proj (\mls A_\bullet )\rightarrow \mls Y,
$$
where $\mls Proj (\mls A_\bullet ):= [\Sp (\mls A_\bullet -\{0\})/\mathbf{G}_m]$ is the stacky proj construction.  For any morphism $T\rightarrow \mls Y$ from a quasi-compact scheme  the base change $\mls X\times _{\mls Y}T$ is the stacky proj $\mls Proj _T(\mls A_T)$ of the pullback of $\mls A$ to $T$ where the canonical line bundle $\mls O_{\mls Proj_T(\mls A_T)}(1)$ is faithful and some power of it descends  to an ample invertible sheaf on the coarse space $\text{Proj}(\mls A_T)$.    In particular, the canonical line bundle $\mls O_{\mls X}(1)$ on $\mls X$ has the property that its restriction to every geometric fiber over $\mls Y$ is $\det $-ample, and if $\mls Y$ is a noetherian scheme then $\mls O_{\mls X}(1)$ is relatively det-ample over $\mls Y$.  This discussion applies, in particular, to weighted stacky blowups  \cite[\S 3]{abramovich2020functorial}. 
 Note that the stacky Proj construction yields a cyclotomic stack so the theory of \cite{AH} applies also in this case. 
\end{example}

\section{Stacks of tame stacks}\label{S:stackofstacks}

Fix an integer $r\geq 1$.
For a  scheme $T$ let $\mls S_r^{(2)}(T)$ be the $2$-category whose objects are pairs $(\mc X, \mls E)$, where $\mc X$ is a proper flat tame Artin stack over $T$ and $\mls E$ is a vector bundle of rank $r$ on $\mc X$ which is $\det $-ample in all geometric fibers (see \ref{P:non-noetherian}). 
Morphisms  $(\mc X, \mls E)\rightarrow (\mc X', \mls E')$ in $\mls S_r^{(2)}(T)$ are given by the groupoid of pairs $(f, \rho )$, where $f:\mc X\rightarrow \mc X'$ is an equivalence and $\rho :f^*\mls E'\rightarrow \mls E$ is an isomorphism.  Let $\mls S_r(T)$ denote the $1$-category associated to $\mls S^{(2)}_r(T)$, so the objects of $\mls S_r(T)$ are the same as those of $\mls S^{(2)}_r(T)$ but the morphisms are given by the isomorphisms classes of morphisms in $\mls S^{(2)}_r(T)$. When the index $r$ is clear we will omit it from the notation.


\begin{thm}\label{T:5.1} (i) The map $\mls S^{(2)}(T)\rightarrow \mls S(T)$ is an equivalence of $2$-categories.

(ii) Let $\mls S$ denote the fibered category over the category of schemes whose fiber over $T$ is given by $\mls S(T)$.  Then $\mls S$ is an algebraic stack.
\end{thm}
The proof will be in several steps. Let $\bar f: X \to T$ be the structure morphism.  
\begin{pg}\label{P:6.2b}
First note that $\mls S^{(2)}$ is a $2$-stack in the sense of 
\cite[1.10]{Breen}, by
\cite[Example 1.11 (i)]{Breen}.  We will not use this fact other than to note that it implies the following two facts (which can also be proved directly, essentially reworking the argument of loc. cit.):
\begin{enumerate}
\item [(i)] Given two objects $(\mls X, \mls E), (\mls X', \mls E')\in \mls S^{(2)}(T)$ over a scheme $T$ the fibered category associating to $T'\rightarrow T$ the groupoid of equivalences $(\mls X_T, \mls E_T)\rightarrow (\mls X'_{T'}, \mls E_{T'})$ is a stack for the \'etale topology.  Note that statement (i) in \ref{T:5.1} is equivalent to the statement that this stack is equivalent to a sheaf, so the verification of (i) is local on $T$.
\item [(ii)] If we assume statement (i) in \ref{T:5.1} then it follows that $\mls S$ is a stack for the \'etale topology.
\end{enumerate}
The main consequence we need from this is that it suffices to prove \ref{T:5.1} (i) only for quasi-compact $T$, where we can describe everything more explicitly, and furthermore since $\mls S$ is a stack to prove \ref{T:5.1} (ii) it suffices to show that $\mls S$ admits open substacks $\mls U_i\subset \mls S$ which are algebraic and such that $\coprod _i\mls U_i\rightarrow \mls S$ is surjective. 
\end{pg}
\begin{pg}\label{P:proofstart}
Next we make some general observations about the discussion in \ref{S:recover}. Let $\mls V$ be a $\mathbf{GL}_r$-representation over $T$ with a decomposition as a direct sum of representations $\mls V = \mls V_1 \oplus \mls V_2$, where $\mathbf{GL}_r$ acts via the character $\det^N$ on $\mls V_1$. First, as in \ref{S:recover} there is a rational map $\QP(V) \dashrightarrow \mathbf{P}(\mls V_1)$ identifying $(\det \mls E_{\QP(\mls V)})^{\otimes N}$ with $\mls O_{\mathbf{P}(\mls V_1)}(1)$. In such a situation we define $\QP(\mls V)^{\circ}$ to be the domain of this rational map. 

Second, if $\mls X \subset \QP(\mls V)^\circ$ is a closed substack, then $\mls E_{\QP(\mls V)}$ restricts to a relatively det-ample vector bundle $\mls E$ on $\mls X$ and we also have a commuting diagram as in \eqref{eq:p17.1}, in particular a morphism $i': X \to \mathbf{P}(\mls V_1)$. If we define $\mls O_X(1):= i'^* \mls O_{\mathbf{P}(\mls{V}_1)}(1)$, then $(\det \mls E)^{\otimes N}$ is identified with $\mls O_X(1)$ and there is a canonical morphism
\begin{equation}\label{eq:canon1}
\mls V_1 \to \bar f_*\mls O_X(1).
\end{equation}
We also get a canonical morphism 
\begin{equation}\label{eq:canon2}
\mls V_2 \to \bar f_*(\pi_*S^m(\mls E^{\oplus r}) (1)).
\end{equation}
as in \eqref{eq:p17.2}.

\end{pg}

\begin{pg}
Given integers $m_0, N >0$, define $\mls S^{(2)}_{N, m_0}(T)$ to be the full 2-subcategory of $\mls S^{(2)}(T)$ with objects $(\mls X, \mls E)$ such that $(\det \mls E)^{\otimes N}$ on $\mls X$ descends to a line bundle $\mls O_X(1)$ on $X$ and conditions (a)-(c) in the proof of \ref{P:non-noetherian} hold.  In particular, from \ref{P:non-noetherian} we see that an object $(\mls X, \mls E)$ of $\mls S^{(2)}_{N, m_0}(T)$ has a canonical immersion (note that to verify that the natural map is an immersion can be verified locally on $T$)
\[
\mls X \hookrightarrow \QP(\mls V) \quad \quad \quad \quad \mls V = \bar f_*\mls O_X(1) \oplus \big ( \oplus_{1 \leq m \leq m_0} \bar f_*(\pi_*S^m(\mls E^{\oplus r}) (1)) \big).
\]
If we set $\mls V_1 = \bar f_*\mls O_X(1)$ then this immersion factors through $\QP(\mls V)$.
Observe that the action of $\mathbf{GL}_r$ on $\mls V$ is polynomial of degree $\leq m_0$.

If $\mls V$ is any $\mathbf{GL}_r$-representation over $S$ with a subrepresentation $\mls V_1 \subset \mls V$ where $\mathbf{GL}_r$ acts via the character $\det^N$, then as in \ref{S:recover} there is a rational map $\QP(V) \dashrightarrow \mathbf{P}(\mls V_1)$ identifying $(\det \mls E)^{\otimes N}$ with $\mls O_{\mathbf{P}(\mls V_1)}(1)$. In such a situation we define $\QP(\mls V)^{\circ}$ to be the domain of this rational map. Using this notation, if we set $\mls V_1 =\bar f_*\mls O_X(1) $ in the above situation then the canonical immersion factors via $\mls X \hookrightarrow \QP(\mls V)^{\circ}$.
\end{pg}

\begin{pg} 
The next lemma and \cite[6.1]{Gtorsors} imply that the condition that $\det(\mls E)^{\otimes N}$ descends to $X$ is an open condition on $\mls X$.
\begin{lem}\label{L:6.6}
Let $\mls X$ be a tame stack with coarse space $X$ and let $\mls L$ be a line bundle on $\mls X$. Then there is an open substack $\mls U \subset \mls X$ such that a geometric point $\bar x \to \mls X$ factors through $\mls U$ if and only if the stabilizer action on $\mls L_{\bar x}$ is trivial. 
\end{lem}
\begin{proof}
We may work locally on $X$. Shrink $X$ first so that $\mls L^{\otimes M}$ is descends to the coarse space for some integer $M$, and then shrink $X$ again so that $\mls L^{\otimes M}$ is trivial. Fixing a trivialization we get a reduction of the $\mathbf{G}_m$-torsor corresponding to $\mls L$ to a $\bmu_M$-torsor $\mls P$. It now suffices to prove the analog of the lemma for $\mls P$, but this is \cite[6.6]{Gtorsors}.
\end{proof}
It follows that the condition that an object $(\mls X, \mls E)\subset \mls S^{(2)}(T)$ lies in $\mls S^{(2)}_{N, m_0}(T)$ is an open condition on $T$ and for quasi-compact $T$ we have 
$$
\text{colim}_{N, m_0}\mls S^{(2)}_{N, m_0}(T) = \mls S^{(2)}(T).
$$
Therefore as argued in \ref{P:6.2b} if we prove that $\mls S^{(2)}_{N, m_0}(T)$ is equivalent to a $1$-category and that the resulting stack $\mls S_{N, m_0}$ is algebraic then we obtain \ref{T:5.1}.  We prove these statements by describing $\mls S^{(2)}_{N, m_0}(T)$ more explicitly.
\end{pg}


\begin{pg}
To do so, fix $N$ and $m_0$ and introduce some auxiliary stacks:  
\begin{enumerate}
    \item  For an integer $a>0$ let $\Sigma _a$ denote the stack whose fiber over a scheme $T$ is the groupoid of vector bundles on $T$ of rank $a$ with polynomial representation of $\mathbf{GL}_r$ of degree $\leq m_0$.
    \item  For integers $a, b>0$ let $\mls H_{a, b}$ denote the stack whose fiber over a scheme $T$ is the groupoid of triples $(\mls V_1, \mls V_2, i)$, where $\mls V_1$ is
    a vector bundle  of rank $b$ (which we view as having action of $\mathbf{GL}_r$ by $\det ^N$),
     $\mls V_2$ is a vector bundle of rank $a$ on $T$ with polynomial action of $\mathbf{GL}_r$ of degree $\leq m_0$, and $i:\mls X\hookrightarrow \QP (\mls V_1\oplus \mls V_2)^\circ $ is a closed substack, proper, locally finitely presented,  and flat over $T$. Note that there is a morphism $\mls H_{a, b} \to \mls S^{(2)}$ sending $(\mls V_1, \mls V_2, i)$ to $(\mls X, i^*\mls E_{\QP(\mls V)}).$
    \item  $\mls H_{a, b}'\subset \mls H_{a, b}$ is defined to be the substack of triples $(\mls V_1, \mls V_2, i)$ for which the induced pair $(\mls X, i^*\mls E_{\QP(\mls V)})$ is an object of $\mls S^{(2)}_{N, m_0}$ and the natural maps \eqref{eq:canon1} and \eqref{eq:canon2} are isomorphisms of $\mathbf{GL}_r$-representations.
\end{enumerate}
To prove \ref{T:5.1} it then suffices to prove the following:

\begin{lem} (i) The stacks $\Sigma _a$, $\mls H_{a, b}$, and $\mls H_{a, b}'$ are algebraic stacks locally of finite type over $\mathbf{Z}$ and the inclusion $\mls H_{a, b}'\subset \mls H_{a, b}$ is an open immersion.

(ii) For every scheme $T$ the natural map $\mls S^{(2)}_{N, m_0}(T)\rightarrow \sqcup_{a, b \in \mathbf{Z}_{\geq 0} }\mls H_{a, b}'(T)$ is an equivalence. 
\end{lem}
\begin{proof}
    To see that $\Sigma _a$ is algebraic, note that the vector bundles of rank $a$ are classified by the algebraic stack $B\mathbf{GL}_a$, and the data of a polynomial $\mathbf{GL}_r$-action of degree $\leq m_0$ on a vector bundle $\mls V$ on a scheme $T$ is classified by a map $\mls V\rightarrow \mls V\otimes _{\mls O_T}\mls O_T[x_{ij}]^{\leq m_0}_{1\leq i,j\leq r}$ satisfying various closed conditions.  From this the statement about $\Sigma _a$ follows.

    Consider the stack $ B\mathbf{GL}_b \times \Sigma _a$ classifying pairs $(\mls V_1, \mls V_2)$ consisting of a rank $b$ vector bundle $\mls V_2$ and a rank $a$ polynomial representation of $\mathbf{GL}_r$ of degree $\leq m_0$, and let $\QP \rightarrow  B\mathbf{GL}_b\times \Sigma _a$ denote the algebraic stack obtained by the construction $(\mls V_1, \mls V_2)\mapsto \QP (\mls V_1\oplus \mls V_2)$ applied to the universal pair.  We can then consider the relative Hilbert functor $\widetilde {\mls H}$ which to any $T\rightarrow  B\mathbf{GL}_b \times \Sigma _a$ corresponding to a pair $(\mls V_1, \mls V_2)$ associates the set of proper flat finitely presented  closed substacks of $\QP (\mls V_1\oplus \mls V_2)^{\circ}$.  By \cite[1.1]{OS03} (which generalizes by the same argument to tame stacks) the stack $\widetilde {\mls H}$ is algebraic.  Since the condition that a stack is tame is an open condition on the base it follows that $\mls H_{a, b}\subset \widetilde {\mls H}$ is also an algebraic stack.  

    The inclusion $\mls H'_{a, b}\subset \mls H_{a, b}$ is an open substack because the condition that a pair $(\mls X, \mls E)\in \mls S^{(2)}$ defines an object of $\mls S^{(2)}_{N, m_0}$ is open, and because the conditions that the canonical morphisms \eqref{eq:canon1} and \eqref{eq:canon2} be isomorphisms is also open.    This concludes the proof of (i).

    Statement (ii) follows from \ref{S:recover} and \ref{P:proofstart}.
\end{proof}

This completes the proof of \ref{T:5.1}. \qed 
\end{pg}

\begin{rem} Let $\mls S_N \subset \mls S$ denote the open substack of pairs $(\mls X, \mls E)$ for which $\det(\mls E)^{\otimes N}$ descends to a line bundle on $X$, and let $\mls Pol$ denote the stack of polarized schemes considered in \cite[\href{https://stacks.math.columbia.edu/tag/0D4X}{Tag 0D4X}]{stacks-project}. There is a natural morphism of algebraic stacks
$$
\mls S_N\rightarrow \mls Pol, \ \ (\mls X, \mls E)\mapsto (X, \pi _*\text{det}(\mls E)^{\otimes N}).
$$
\end{rem}

\section{The stack of tame orbicurves}\label{S:section8}

In this section we apply the preceding theory to prove that the stack parameterizing tame orbicurves is algebraic.

\subsection{Tame Artin curves are projective}

\begin{defn}
    Let $k$ be a field. A \textit{tame Artin curve} over $k$ is an algebraic stack $\mls C$ over $k$ such that
    \begin{enumerate}
        \item $\mls C$ is a finite type  tame Artin stack over $k$, and
        \item $\mls C$ is proper and geometrically connected over $k$ and has dimension 1.
    \end{enumerate}
    The coarse space of a tame Artin curve is a curve in the sense of \cite[\href{https://stacks.math.columbia.edu/tag/0D4Z}{Tag 0D4Z}]{stacks-project}.
    We say that a tame Artin curve is a \emph{tame orbicurve} if it contains a dense open substack which is a scheme. 
\end{defn}

Note that we allow a tame Artin curve to be non-integral, and even non-reduced. We record the following vanishing result.

\begin{prop}\label{prop:vanishing on a curve}
    If $\mls C$ is a tame Artin curve over a field $k$ and $\mls E$ is a coherent sheaf on $\mls C$, then ${H}^q(\mls C,\mls E)=0$ for all $q\geq 2$.
\end{prop}
\begin{proof}
    Let $\pi:\mls C\to C$ be the coarse moduli space of $\mls C$. Since $C$ has dimension $1$, the cohomology of any coherent sheaf on $C$ vanishes in degrees $\geq 2$. As $\mls C$ is tame, the pushforward $\pi_*$ is exact, and so we have ${H}^q(\mls C,\mls E)={H}^q(C,\pi_*\mls E)=0$ for any $q\geq 2$.
\end{proof}

The following result generalizes the fact that a proper curve over $k$ admits an ample line bundle.
\begin{prop}\label{thm:existence of an ample vector bundle}
    If $\mls C$ is a tame Artin curve over a field $k$, then there exists a det-ample vector bundle on $\mls C$.
\end{prop}
\begin{proof}
By \cite[Corollary 60]{mathur2021resolution} the stack $\mls C$ admits a faithful vector bundle. Therefore  the result follows from \ref{E:3.6} and the fact that the coarse space of $\mls C$ is a proper curve over a field and hence has an ample line bundle  (combine \cite[\href{https://stacks.math.columbia.edu/tag/0A26}{Tag 0A26}]{stacks-project} and \cite[\href{https://stacks.math.columbia.edu/tag/0ADD}{Tag 0ADD}]{stacks-project}).
\end{proof}

\begin{defn}
A \textit{family of tame Artin curves} (resp. \emph{family of tame orbicurves}) over a scheme $T$ consists of an algebraic stack $\mls C$ and a flat proper morphism $f:\mls C\to T$ which is locally of finite presentation such that for every geometric point $\bar t:\Sp (k)\to T$ the fiber $\mls C_{\bar t}$ is a tame Artin curve (resp. tame orbicurve) over $k$.
\end{defn}

\begin{lem}\label{lem:family of tame artin stacks is a tame artin stack}
    If $T$ is a scheme and $f:\mls C\to T$ is a family of tame Artin curves, then $\mls C$ is a tame Artin stack.
\end{lem}
\begin{proof}
Since the geometric fibers of $\mls C$ are tame, by \cite[3.2]{tame} the automorphism group scheme of every geometric point of $\mls C$ is finite and linearly reductive. Hence using \cite[3.2]{tame} again, to show that $\mls C$ is tame it is enough to show that its inertia stack $I_{\mls C}$ is finite over $\mls C$. To show this, we observe that as $f:\mls C\to T$ is proper, the diagonal
    \[
        \Delta_{\mls C/T}:\mls C\to\mls C\times_T\mls C
    \]
    is also proper. As $T$ is a scheme, the relative inertia stack $ I_{\mls C/T}$ is equal to $ I_{\mls C}$, and it follows that the map $I_{\mls C}\to\mls C$ is proper. Since geometric points of $\mls C$ have finite stabilizers, the morphism $ I_{\mls C}\to\mls C$ is quasi-finite, hence finite.
\end{proof}

\subsection{The stack of tame orbicurves}

\begin{pg}\label{P:5.8}
    Let $\mls M^{(2)}$ denote the $2$-category  fibered in 2--groupoids over the category of schemes whose objects are pairs $(\mls C,T)$ where $T$ is a scheme and $\mls C\to T$ is a family of tame orbicurves over $T$. A morphism $(\mls C,T)\to(\mls C',T')$ in $\mls M^{(2)}$ is a pair $(f,g)$, where $f:\mls C\to\mls C'$ is a morphism of stacks and $g:T\to T'$ is a morphism of schemes such that the diagram
    \[
        \begin{tikzcd}
            \mls C\arrow{d}\arrow{r}{f}&\mls C'\arrow{d}\\
            T\arrow{r}{g}&T'
        \end{tikzcd}
    \]
    cartesian. A 2--morphism $(f_0,g_0)\to (f_1,g_1)$ between two such 1--morphisms  exists only if $g_0 = g_1$, in which case it  consists of a  natural transformation $\theta:f_0\xrightarrow{\sim}f_1$. 
\end{pg}

\begin{pg}
    Associated to $\mls M^{(2)}$ is a category $\mls M$ fibered in groupoids over the category of schemes.
    Objects of $\mls M$ are pairs $(\mls C,T)$ where $T$ is an scheme and $\mls C\to T$ is a family of tame orbicurves over $T$, and a morphism $(\mls C,T)\to (\mls C',T')$ in $\mls M$ between two objects is an isomorphism class of 2--Cartesian diagrams
    \[
        \begin{tikzcd}
            \mls C\arrow{d}\arrow{r}&\mls C'\arrow{d}\\
            T\arrow{r}&T'.
        \end{tikzcd}
    \]

    \begin{prop}\label{P:8.11}
        The natural functor $\mls M^{(2)}\to\mls M$ is an equivalence. 
    \end{prop}
    \begin{proof}
    This follows from \cite[4.2.3]{AVmaps}, which implies that a $2$-morphism $(f_0,g_0)\to (f_1,g_1)$ as in \ref{P:5.8} is unique if it exists.
    \end{proof}
\end{pg}

\begin{prop}\label{P:7.11} $\mls M$ is a stack.
\end{prop}
\begin{proof}
Let $\widetilde{\mls M}^{(2)}(T)$ denote the 2-groupoid of stacks over $T$. We claim that the condition that an object $\mls X$ of $\widetilde{\mls M}^{(2)}(T)$ lies in $\mls M^{(2)}(T)$ is open on $T$. Then the proposition follows from the fact that the fibered 2-category assembled from $\widetilde{\mls M}^{(2)}(T)$ is a 2-stack by \cite[Example 1.11 (i)]{Breen}.

We now prove the claim, First, by \cite[\href{https://stacks.math.columbia.edu/tag/0D4R}{Tag 0D4R}]{stacks-project} 
        the condition of being relative dimension $1$, which can be checked on coarse spaces, is an open and closed condition.  We let $\mls M_1$ denote the corresponding stack of 1-dimensional tame stacks and let $\mls C \to \mls M_1$ denote the universal object. Second, the inertia stack $I_{\mls C}\rightarrow \mls C$ is finite over $\mls C$ and admits a section (the zero-section), and therefore corresponds to a coherent sheaf of algebras $\mls O_{I_{\mls C}}\simeq \mls O_{\mls C}\oplus \mls F$.  Let $\mls Z\subset \mls C$ be the support of $\mls F$ so that the complement of $\mls Z$ is the maximal open substack of $\mls C$ representable by an algebraic space. Then $\mls Z \to \mls M_1$ is proper and by upper semicontinuity of fiber dimensions there is an open substack of $\mls M_1$ where the fibers of $\mls Z$ are zero-dimensional. This is precisely the locus where geometric fibers of $\mls C \to \mls M_1$ have a dense subscheme. 
\end{proof}

     We call $\mls M$ the \emph{stack of tame orbicurves}. 

\begin{thm}\label{T:8.13} $\mls M$ is an algebraic stack.
\end{thm}

Before giving the proof we record the following general result.

    \begin{lem}\label{lem:covering stack lemma}
        Let $S$ be a scheme. Let $\mls U$ be an algebraic stack over $S$ and let $\mls M$ be a stack over $S$. Suppose that there exists a morphism $\pi:\mls U\to\mls M$ of stacks over $S$ such that
        \begin{itemize}
            \item [(i)] for any morphism $T\to\mls M$ from an $S$--scheme $T$, the fiber product $\mls U\times_{\mls M}T$ is an algebraic stack and the morphism $\mls U\times_{\mls M}T\to T$ is locally of finite presentation and surjective, and
            \item [(ii)] $\pi$ is formally smooth, in the sense that given a square-zero extension $A'\twoheadrightarrow A$ of rings and a 2--commutative diagram
        \begin{equation}\label{eq:deformation diagram}
        \begin{tikzcd}
            \Sp (A)\arrow{r}\arrow[hook]{d}&\mls U\arrow{d}{\pi}\\
            \Sp (A')\arrow{r}\arrow[dashed]{ur}&\mls M
        \end{tikzcd}
        \end{equation}
        of solid arrows, there exists a dashed arrow rendering the resulting diagram 2--commutative.
        \end{itemize} 
        Then $\mls M$ is an algebraic stack.
    \end{lem}
    \begin{proof}
        Let $U\rightarrow \mls U$ be a smooth surjective morphism with $U$ a scheme.  Then the composition $U\rightarrow \mls U\xrightarrow{\pi} \mls M$ is representable, smooth, and surjective. 
        Indeed given any morphism $T\rightarrow \mls M$ from a scheme $T$ we have 
        $$
U\times _{\mls M}T\simeq U\times _{\mls U}(\mls U\times _{\mls M}T).
$$
By assumption (i) $\mls U\times _{\mls M}T$ is an algebraic stack which by (ii) is smooth over $T$, and since $U\rightarrow \mls U$ is smooth and surjective it follows that $U\times _{\mls M}T$ is also smooth and surjective over $T$. Moreover $\mls U \times_{\mls M} T \to \mls U$ is representable (using \cite[\href{https://stacks.math.columbia.edu/tag/04Y5}{Tag 04Y5}]{stacks-project} and the fact that $T$ is an algebraic space), so $U \times_{\mls M} T$ is an algebraic space. In particular, setting $R:= U\times _{\mls M}U$ we get a groupoid in algebraic spaces and an equivalence $[U/R]\simeq \mls M$.  The result therefore follows from \cite[\href{https://stacks.math.columbia.edu/tag/04TK}{Tag 04TK}]{stacks-project}.  
    \end{proof}

    \begin{proof}[Proof of \ref{T:8.13}]
       For a positive integer $r$, let $\mls S_r$ denote the stack from \S\ref{S:stackofstacks}. Let $\mls M^{\mls E_r} \subset \mls S_r$ denote the full substack where $\mls X \to T$ is a family of tame orbicurves. As in the proof of \ref{P:7.11}  the conditions defining $\mls M^{\mls E_r}$ are open, so if we set $\mls M^{\mls E}:= \sqcup_{r \in \mathbf{Z}>0} \mls M^{\mls E_r}$, then we have from \ref{T:5.1} that $\mls M^{\mls E}$ is an algebraic stack. Consider the forgetful morphism
        \[
            p:\mls M^{\mls E}\to\mls M
        \]
        given on objects by $(\mls C,\mls E)\mapsto \mls C$. We will verify the conditions of Lemma \ref{lem:covering stack lemma} for $p$. Suppose given a scheme $T$ and a morphism $T\to\mls M$, corresponding to a family $\mls C\to T$ of tame Artin curves over $T$. Let $\mls U=\mls M^{\mls E}\times_{\mls M}T$ be the fiber product, so that we have a 2--Cartesian diagram
        \[
            \begin{tikzcd}
                \mls U\arrow{r}\arrow{d}[swap]{p_T}&\mls M^{\mls E}\arrow{d}{p}\\
                T\arrow{r}&\mls M.
            \end{tikzcd}
        \]
        Explicitly, $\mls U$ is the stack over $T$ whose fiber over a $T$--scheme $T'$ is the groupoid of vector bundles on $\mls C'=\mls C\times_TT'$ that are det-ample in every geometric fiber. 

\begin{lem} $\mls U$ is an algebraic stack locally of finite presentation over $T$.
\end{lem}
\begin{proof}
The stack of all vector bundles on $\mls C$ is algebraic by \cite[1.2]{HallRydh}.  Since the condition that a vector bundle is $\det $-ample in every geometric fiber is open by \ref{P:non-noetherian}, the lemma follows. 
\end{proof}
       By \ref{thm:existence of an ample vector bundle}, the morphism $p$  is surjective on $k$--points for any field $k$, and so is surjective as a morphism of stacks. It remains to prove that $p$ is formally smooth. Consider a square--zero extension $A'\twoheadrightarrow A$ of rings with kernel $I\subset A'$ and a 2--commutative diagram 
    \begin{equation}\label{eq:deformation diagram2}
        \begin{tikzcd}
            \Sp (A)\arrow{r}{x_A}\arrow[hook]{d}&\mls M^{\mls E}\arrow{d}{\pi}\\
            \Sp (A')\arrow{r}{x_{A'}}\arrow[dashed]{ur}&\mls M
        \end{tikzcd}
    \end{equation}
    of solid arrows. The map $x_{A'}$ corresponds to a family $\mls C_{A'}$ of tame orbicurves over $\Sp (A')$, and the map $x_A$ corresponds to a family $\mls C_{A}$ of tame orbicurves over $\Sp (A)$ equipped with a vector bundle $\mls E_A$ that is det-ample in every geometric fiber and a 2--Cartesian diagram
    \[
        \begin{tikzcd}
            \mls C_A\arrow{d}\arrow[hook]{r}&\mls C_{A'}\arrow{d}\\
            \Sp (A)\arrow[hook]{r}&\Sp (A').
        \end{tikzcd}
    \]
    There is a canonical obstruction class
    \[
        \mathrm{ob}\in{H}^2(\mls C_A,\sEnd(\mls E_A)\otimes I)
    \]
    whose vanishing is equivalent to the existence of a flat deformation of $\mls E_A$ to a vector bundle on  $\mls C_{A'}$ (this is classical; a reference in a much more general setting is \cite[IV, 3.1.5]{Illusie}).  By \ref{P:non-noetherian} any such deformation is automatically det-ample in every geometric fiber.  Thus, the class $\mathrm{ob}$ vanishes if and only if there exists a dashed arrow rendering the diagram~\eqref{eq:deformation diagram2} 2--commutative. 
    Using that $\Sp(A)$ is affine and arguing as in the proof of \ref{prop:vanishing on a curve} we see that ${H}^2(\mls C_A,\sEnd(\mls E_A)\otimes I)$ vanishes, and hence ob is always zero.
    This completes our verification that $p$ is formally smooth.
    \end{proof}

\begin{rem}
    Let $\mls C\to T$ be a family of tame orbicurves with coarse moduli space $\mls C\to C$. Then $C\to T$ is again flat and proper, and the moreover the formation of the coarse moduli space is compatible with arbitrary base change on $T$. Thus, letting $\mls M^c$ denote the stack of curves in the sense of \cite[\href{https://stacks.math.columbia.edu/tag/0DMJ}{Tag 0DMJ}]{stacks-project}, the association $\mls C\mapsto C$ defines a morphism $\mls M\to\mls M^c.$
\end{rem}

\begin{rem}
    It is natural to attempt to remove the tameness assumption, and consider the stack parameterizing proper stacks of dimension $1$ with generically trivial stabilizers. We do not know whether this stack is algebraic. We point out that the crucial vanishing of coherent cohomology in degrees $\geq 2$ fails in the wild case.
\end{rem}

\begin{appendix}
    \section{Ample line bundles on algebraic spaces}
The results in this appendix are presumably well-known but we include them here for lack of a reference.

    Let $R$ be a noetherian ring.  If $X$ is an algebraic space over $R$ and $\mc O_X(1)$ is an invertible sheaf on $X$, then for an integer $n$ and coherent sheaf $\mls F$ on $X$ write $\mls O_X(n)$ for $\mls O_X(1)^{\otimes n}$ and $\mls F(n)$ for $\mls F\otimes \mls O_X(n)$.

When $X$ is a scheme, the following result follows from \cite[\href{https://stacks.math.columbia.edu/tag/01Q3}{Tag 01Q3}]{stacks-project}.

    \begin{thm}\label{T:A.1} Let $X$ be a separated algebraic space of finite type over $R$. Assume that for every coherent sheaf $\mls F$ on $X$ there exists an integer $n_0$ such that $\mls F(n)$ is generated by global sections for $n\geq n_0$.  Then $X$ is a scheme and $\mls O_X(1)$ is ample on $X$.
    \end{thm}
    \begin{proof}
        First of all, replacing $\mls O_X(1)$ by $\mls O_X(n)$ for suitable $n$ we may assume that $\mls O_X(n)$ is generated by global sections for all $n\geq 1$.  Let $S$ be the graded ring $\oplus _{n\geq 0}H^0 (X, \mls O_X(n))$ so we obtain a morphism $\rho :X\rightarrow \text{Proj}(S)$.  We claim that this morphism is quasi-finite.  If we show this then it follows that $X$ is a scheme, by \cite[\href{https://stacks.math.columbia.edu/tag/0418}{Tag 0418}]{stacks-project}, and therefore $\mc L$ is ample \cite[\href{https://stacks.math.columbia.edu/tag/01Q3}{Tag 01Q3}]{stacks-project}.

        To prove that $\rho $ is quasi-finite 
         note that by \cite[\href{https://stacks.math.columbia.edu/tag/0BBN}{Tag 0BBN}]{stacks-project} there exists a finite collection $T_i\subset X$ of disjoint locally closed subspaces with each $T_i$ a scheme and covering $X$ ($X$ admits a stratification by schemes).  
        To prove the quasi-finiteness of $\rho $ it then  suffices to show that its restriction $\rho :T_i\rightarrow \text{Proj}(S)$ is quasi-finite for each $i$.  For this it suffices, in turn, to show that the  open sets $X_s\cap T_i\subset T_i$, as $s$ varies over homogeneous elements of $S$, generate the topology of $T_i$.  For this statement, let $U\subset T_i$ be an open set and let $Z\subset T_i$ be its complement.  Let $\overline Z\subset X$ be the closure of $Z$ in $X$, viewed as a subscheme with the reduced induced structure.  Fix also a point $y\in Y$.  If $I\subset \mls O_X$ is the ideal sheaf of $\overline Z$ then we can find an integer $n>0$ such that $I(n)$ is generated by global sections and therefore there exists a section $s\in H^0(X, \mls O_X(n))$ whose zero locus contains $Z$ and which is nonzero at $y$.  For such a section $s$ we have $y\in X_s\cap T_i\subset U$.
    \end{proof}

The following result is similar to \cite[Proof of Theorem 1]{Kubota}, which a priori applies when $X$ is a scheme.
\begin{prop}\label{P:A.2b} Let $f:P\rightarrow X$ be a proper morphism of algebraic spaces of finite type over $R$
and  let $\mls O_P(1)$ be an ample invertible sheaf on $P$.  Then for every coherent sheaf $\mls F$ on $P$ there exists an integer $n_0$ such that for all $n\geq n_0$ the sheaf $f_*\mls F(n)$ is generated by global sections.
\end{prop}
\begin{proof}
    Note that it suffices to prove the proposition after replacing $\mls O_P(1)$ by $\mls O_P(d)$ for $d>0$.  Indeed if the result holds with $\mls O_P(d)$ then there exists an integer $m_0$ such that for all $m\geq m_0$ the sheaves 
    $$
    f_*\mls F(md), \;f_*\mls F(md+1), \;\dots, \;f_*\mls F(md+(d-1))
    $$
    are generated by global sections.  It follows that for all $n\geq m_0d$ the sheaf $f_*\mls F(n)$ is generated by global sections.  Replacing $\mls O_P(1)$ by $\mls O_P(d)$ we may therefore assume that we have an immersion $i:P\hookrightarrow \mathbf{P}^N_R$ for some $N$ with $i^*\mls O_{\mathbf{P}^N_R}(1)\simeq \mls O_P(1)$.  Let $i':P\hookrightarrow \mathbf{P}^N_X$ be the induced immersion, which is closed since $f$ is proper.  If $g:\mathbf{P}^N_X \to X$ is the projection then we have 
    $$
    f_*\mls F(n)\simeq g_*((i'_*\mls F)\otimes \mls O_{\mathbf{P}^N_X}(n)),
    $$
    from which it follows that it suffices to prove the proposition with $P = \mathbf{P}^N_X$ and $\mls O_P(1) = \mls O_{\mathbf{P}^N_X}(1)$.  Note that in this case $f_*\mls O_{P}(n) = H^0(\mathbf{P}^N_R, \mls O_{\mathbf{P}^N_R}(n))\otimes _R\mls O_X$; in particular, it is generated by global sections for $n>0$.  Now choose an integer $s_0$ such that for all $n\geq s_0$ the sheaf $\mls F(n)$ on $P$ is generated by global sections, and fix a surjection $a:\xymatrix{\mls O_P^{\oplus r}\ar@{->>}[r]& \mls F(s_0)}$ for some $r$, and let $\mls K$ denote the kernel of $a$ so that for all $n\geq s_0$ we have an exact sequence
    $$
    0\rightarrow \mls K(n-s_0)\rightarrow \mls O_P(n-s_0)^{\oplus r}\rightarrow \mls F(n)\rightarrow 0.
    $$
    Now choose $n_0\geq s_0$ so that $R^1f_*\mls K(n-s_0)=0$ for all $n\geq n_0$ (this is possible by \cite[\href{https://stacks.math.columbia.edu/tag/0B5U}{Tag 0B5U}]{stacks-project}).  Then for $n\geq n_0$ the map
    $$
    (f_*\mls O_P(n-s_0))^{\oplus r}\rightarrow f_*\mls F(n)
    $$
    is surjective, and the left side is globally generated.
\end{proof}

\section{Affine GIT over a general base}\label{A:appendixB}

\begin{pg}
Let $R$ be a noetherian ring and let $G$ be a reductive group scheme over $R$
\cite[Expos\'e XIX]{SGA3reductive}.  Recall that this means that $G$ is an affine smooth $R$-group scheme all of whose geometric fibers are connected and reductive.  By \cite{Seshadri} the basic constructions of geometric invariant theory carry through for actions of $G$ on quasi-projective $R$-schemes, at least in the case when $R$ is of finite type over a universally Japanese ring.  
In his theory of adequate moduli spaces, Alper generalizes some of these constructions to arbitrary noetherian rings \cite{Alperadequate}. In this appendix we restate some of Alper's work in traditional GIT language and explain how for some results we can even drop the noetherian hypothesis.
\end{pg}

\subsection{Affine GIT over $R$} Let $X = \Sp (A)$ be an affine $R$-scheme with action of $G$, let $Y$ denote $\Sp (A^G)$, let $\mls X$ denote the stack quotient $[X/G]$, and let
\[
X \xrightarrow{q} \mls X \xrightarrow{p} Y
\]
denote the canonical morphisms.
It follows from \cite[\href{https://stacks.math.columbia.edu/tag/02FZ}{Tag 02FZ}]{stacks-project} that there exists an open substack $\mls X_0\subset \mls X$ such that a geometric point $\bar x\rightarrow \mls X$ factors through $\mls X_0$ if and only if the stabilizer group scheme of $\bar x$ is finite.
  Let $\mls Z\subset \mls X$ be the complement of $\mls X_0$ with the reduced substack structure.

 By \cite[9.1.4]{Alperadequate} the map $p:\mls X\rightarrow Y$ is an adequate moduli space in the terminology of loc. cit., and therefore so is the base change $\mls X_U\rightarrow U$ for any open subset $U\subset Y$.  It follows from this and \cite[5.3.1]{Alperadequate}  that the image $Z\subset Y$ of $\mls Z$ is closed.  Let $U\subset Y$ be the complement of $Z$ and let $\mls U\subset \mls X$ be the preimage.

\begin{prop}\label{P:B.3b} The stack $\mls U$ has finite diagonal.
\end{prop}
\begin{proof} This follows from \cite[8.3.2]{Alperadequate}. Indeed by loc. cit. the map $\mls U \rightarrow U$ to its coarse moduli space is separated and therefore the diagonal map $\mls U\rightarrow \mls U\times _U\mls U$ is finite. 
 Since the map $\mls U\times _U\mls U\rightarrow \mls U\times _{\Sp (R)}\mls U$ is a closed immersion, being a base change of the closed immersion $U\rightarrow U\times _{\Sp (R)}U$ (since $U/R$ is separated) it follows that the diagonal map $\mls U\rightarrow \mls U\times _{\Sp (R)}\mls U$ is finite.
\end{proof}

\begin{pg}\label{P:B.4} The relationship with affine GIT is as follows.  Recall \cite[Definition 1 on p. 252]{Seshadri} that a geometric point $\bar x\rightarrow X$ is \emph{stable} if the $G$-orbit of $\bar x$ in $X\times _{\Sp (R)}\bar x$ is closed of dimension equal to the dimension of $G_{\bar x}$ (this last condition on the dimension of the orbit is equivalent to the condition that $\bar x$ has zero-dimensional stabilizer). 
\end{pg}

\begin{prop}\label{P:B.3} A geometric point $\bar x\rightarrow X$ is stable if and only if it factors through $X^s:= q^{-1}(\mls U)\subset X$.  In particular, $X^s$ is an open subset of $X$ and the quotient $[X^s/G]$ has finite diagonal with coarse space an open subset of $\Sp (A^G)$.
\end{prop}
\begin{proof}
By definition a geometric point $\bar x\rightarrow X$ factors through $q^{-1}(\mls X_0)$ if and only if the stabilizer group scheme is finite. Since $\mls U$ is contained in $\mls X_0$, it therefore suffices to show that a geometric point $\bar x\rightarrow \mls X_0$ factors through $\mls U$ if and only if the orbit of $\bar x$ in $X\times _{\Sp (R)}\bar x$ is closed.

For this we use the property \cite[5.3.1 (4)]{Alperadequate} that 
two geometric $k$-points have the same image in $Y$ if and only if the closures in $X \times_{\Sp(R)} \Sp(k)$ have nonempty intersection. So the orbit of the given point $\bar x$ is closed if and only if any other geometric point $\bar x'$ with the same image in $Y$ contains $\bar x$ in its orbit closure. But since $\bar x$ has finite stabilizer it must be in the orbit of $\bar x'$: if not, the orbit of $\bar x$ would have strictly smaller dimension than the orbit of $\bar x'$, and hence the stabilizer of $\bar x$ would have strictly larger dimesnion than the stabilizer of $\bar x'$. Since the stabilizer of $\bar x$ is finite this is not possible, so $\bar x'$ is in the same orbit as $\bar x$ and in particular has finite stabilizer. 

We have shown that $\bar x \to \mls X_0$ has closed orbit if and only if every other point in the same fiber over $Y$ is also in $\mls X_0$; i.e., if and only if $\bar x$ factors through $\mls U$.
\end{proof}

\begin{rem} Statements \ref{P:B.3b} and \ref{P:B.3}  hold more generally with $G$ geometrically reductive in the sense of \cite{Alperadequate}, with the same proofs.  In particular, by \cite[9.7.6]{Alperadequate} the group scheme $G$ could be an extension of a finite group scheme by a reductive group scheme. 
\end{rem}

\subsection{Twisted affine GIT over $R$} Our next goal is to state and prove the analog of \ref{P:B.3} for stability with respect to a character $\chi: G \to \mathbf{G}_{m, R}$. Let $X = \Spec(A)$ be an affine $R$-scheme with action of $G$. For any integer $m$ let $A_{\chi^m} \subset A$ be the $R$-submodule of elements on which $G$ acts through the character $\chi^m$.

\begin{defn}\label{D:B.7} A geometric point $\bar x\rightarrow X=\Sp (A)$ is \emph{stable with respect to} $\chi$ if for some integer $m>0$ there exists $f\in A_{\chi ^m}$ such that $\bar x^*f\neq 0$, the orbit of $\bar x$ in $D(f)\times _{\Sp (R)}\bar x$ is closed, and the stabilizer group scheme of $\bar x$ is finite.
\end{defn}

\begin{rem}\label{R:a}
A point $\bar x \to X$ is stable with respect to $\chi$ if and only if for some $m>0$ there exists $f \in A_{\chi^m}$ such that $\bar x^*f \neq 0$ and $\bar x$ is a point of $D(f)^s$, the locus of stable points in the sense of \ref{P:B.4} for the affine scheme $D(f)$.
\end{rem}
\begin{rem} If $N>0$ is an integer, a geometric point $\bar x\rightarrow X$ is stable with respect to $\chi $ if and only if it is stable with respect to $\chi ^N$.  The ``if'' direction is immediate, and for the ``only if'' direction observe that if $f\in A_{\chi ^m}$ is a section then $f^N$ is an element of $A_{\chi ^{mN}}= A_{(\chi ^N)^m}$ and $D(f) = D(f^N)$.
\end{rem}

\begin{lem}\label{L:B.9} There exists an open subset $X^{s, \chi}
\subset X$ such that a geometric point $\bar x\rightarrow X$ factors through $X^{s, \chi}$ if and only if $\bar x$ is stable with respect to $\chi$.
\end{lem}
\begin{proof}
Define $X^{s,\chi}$ to be the union of $D(f)^s$ where $f$ ranges over all elements of $A_{\chi^m}$ for all $m>0$. Then $X^{s, \chi}$ has the desired properties by \ref{R:a} and \ref{P:B.3} (which says that $D(f)^s$ is open in $D(f)$).

\end{proof}

 Let $S$ denote the graded ring $S:= \oplus _{m\geq 0}A_{\chi ^m}$, and let $S_{(f)}$ denote the subring of the localization $S_f$ consisting of elements of degree zero. 
 
\begin{lem}\label{L:B.10} For $m>0$ and $f\in A_{\chi ^m}$ the natural map
$$
S_{(f)}\rightarrow (A_f)^{G}
$$
is an isomorphism.
\end{lem}
\begin{proof}
The injectivity follows from the fact that the map $S\rightarrow A$ is injective.

For the surjectivity, let $y\in (A_f)^G$ be an element, and choose $M>0$ such that $f^My\in  A_f$ is in the image of $A$. Since $G$ is not linearly reductive, an arbitrary element of the preimage of $f^My$ may not be contained in $A_{\chi^M}$, but we can find one such element as follows. Let $M\subset A$ be a finitely generated $G$-invariant submodule surjecting onto $R\cdot f^My$, and let $M'\subset M$ be the kernel of the induced surjection $M\rightarrow R$.  Since the module  $M'$ has image $0$ in $A_f$ there exists an integer $N>0$ such that $f^NM' = 0$.  The image of $M$ under the map $f^N:A\rightarrow A$ is then a $G$-submodule of $A$ mapping isomorphically to $R\cdot f^{N+M}y$ and is therefore contained in $A_{\chi ^{m(N+M)}}$.  It follows that there exists an element $g\in A_{\chi ^{m(N+M)}}$ such that $g/f^{N+M}\in S_{(f)}$ maps to $y$. 
\end{proof}

\begin{lem}\label{L:B.11} Let $\bar x$ be a geometric point of $X^{s, \chi}$.  Then there exists $m>0$ and $f\in A_{\chi ^m}$ not vanishing at $\bar x$ such that every point of $D(f)$ is stable and the graph of the action map 
\begin{align}
G\times D(f)&\rightarrow D(f)\times D(f)\label{eq:action}\\
(g, y)&\mapsto (y, gy)\notag
\end{align}
is proper.
\end{lem}
\begin{proof}
    By \ref{R:a} there exists $m>0$ and $f\in A_{\chi ^m}$ such that $\bar x$ is a point of $D(f)^s$.  By \ref{P:B.3} the set $D(f)^s$ is the preimage of an open subset $U\subset \Sp ((A_f)^G)$.  It follows that there exists an element $h\in (A_f)^G$ such that $D(h)\subset D(f)^s$ and $\bar x$ has image in $D(h)$.  Now (using \ref{L:B.10}) choose $N>0$ such that $f^Nh$ lifts to an element $\tilde h\in A_{\chi ^{Nm}}$.  
    Then $f\tilde h$ is in $ A_{\chi ^{(N+1)m}}$ and $D(f\tilde h) = D(h)\subset X$.  That is, for a stable point $\bar x\rightarrow X$ we can find an element $f\in A_{\chi ^m}$ for some $m>0$ for which $D(f)^s = D(f)$.  This implies, in particular that the diagonal of the stack quotient $[D(f)/G]$ is finite.  To conclude the proof note that every point $\bar x \to D(f)$ is stable with respect to $\chi$ by \ref{R:a} and the diagram
    $$
    \xymatrix{
    G\times D(f)\ar[r]\ar[d]& D(f)\times D(f)\ar[d]\\
    [D(f)/G]\ar[r]^-\Delta & [D(f)/G]\times [D(f)/G]}
    $$
is cartesian, where the top horizontal map is the graph of the action and the vertical maps are the projections.
\end{proof}

Observe that there is a natural $G$-invariant map $X^{s,\chi}\rightarrow \Sp (S)-\{0\}$ (where we write $0$ for the closed subscheme defined by the ideal $\oplus _{m>0}A_{\chi ^m}\subset S$).  Passing to the quotient by the $G$ and $\mathbf{G}_m$-actions respectively we get a map $\rho :[X^{s,\chi}/G]\rightarrow \text{Proj}(S)$.

\begin{cor}\label{C:B.12} The stack $[X^{s,\chi}/G]$ has finite diagonal with coarse space an open subscheme of $\text{\rm Proj}(S)$.
\end{cor}
\begin{proof}
By \ref{L:B.11} the stack $[X^{s, \chi}/G]$ is covered by open substacks of the form $[D(f)/G]$ where $f \in A_{\chi^m}$ has the property that the action map \ref{eq:action} is proper. This implies that $[D(f)/G]$ has finite diagonal and hence $[X^{s,\chi}/G]$ does as well. From the cartesian diagram
$$
\xymatrix{
[D(f)/G]\ar@{^{(}->}[r]\ar[d]& [X^s/G]\ar[d]^-\rho \\
D_+(f)\ar@{^{(}->}[r]& \text{Proj}(S),}
$$
it follows that the coarse space of $[X^{s,\chi}/G]$ is the open subscheme of $\text{Proj}(S)$ given by the union of the open sets $D_+(f)$ as $f$ ranges over elements of $A_{\chi ^m}$ with all points of $D(f)$ stable.

\end{proof}

\begin{rem} The above discussion extends to the setting of a quasi-projective $R$-scheme with a $G$-linearized invertible sheaf.  However, we do not develop this theory since it will not be used in the article.
\end{rem}

\begin{rem}\label{rem:git-detample} By \cite[\href{https://stacks.math.columbia.edu/tag/01MW}{Tag 01MW}]{stacks-project} the scheme $\text{Proj}(S)$ has an ample line bundle which pulls back to some power of the line bundle on $[X^s/G]$ defined by the character $\chi $.  In particular, if $G = \mathbf{GL}_r$ for some $r$ and $\chi $ is a positive power of the determinant then the vector bundle $\mls E$ on $[X^s/\mathbf{GL}_r]$ associated to the standard representation of $\mathbf{GL}_r$ is $\det $-ample on $[X^s/\mathbf{GL}_r]$ in the sense of \ref{D:ampledef}.
\end{rem}

\subsection{Other characterisations of the stable locus}
Observe (as in \cite[6.1]{Mukaimoduli}) that there is a canonical identification
\begin{equation}\label{eq:identify}
\oplus _{m\geq 0}A_{\chi ^m} \simeq (A[u])^G
\end{equation}
where $G$ acts on $u$ via $\chi ^{-1}$.  Geometrically, we can view $A[u]$ as the coordinate ring of the $G$-scheme $\widetilde{X}:= X\times _{\Sp (R)}\mathbf{A}^1_{R, \chi ^{-1}}$, where $\mathbf{A}^1_{R, \chi ^{-1} }$ denotes the affine line with $G$-action by $\chi ^{-1}$.
Consider the projection $\widetilde{X} \to X$ .

\begin{lem}\label{L:B.15} Let $\bar x\rightarrow X$ be a geometric point, and let $\bar y\rightarrow X\times \mathbf{G}_{m}\subset \widetilde X$ be any lift.

(i) There exists $m>0$ and $f\in A_{\chi ^m}$ such that $\bar x^*f\neq 0$ if and only if the closure of the $G$-orbit of $\bar y$ in $\widetilde X_{\kappa (\bar x)}$ does not meet $X_{\kappa (\bar x)}\times \{0\}\subset \widetilde X_{\kappa (\bar x)}$.

(ii) The point $\bar x$ is stable if and only if the $G$-orbit of $\bar y$ in $\widetilde X_{\kappa (\bar x)}$ is closed and the stabilizer group scheme of $\bar y$ is finite.
\end{lem}
\begin{proof}
Let $s\in \Sp (R)$ be the image of $\bar x$.  The map
$$
\Gamma (\widetilde X, \mls O_{\widetilde X})^G\rightarrow \Gamma (\widetilde X_s, \mls O_{\widetilde X_s})^G
$$
need not be surjective, but is an adequate homeomorphism in the sense of \cite[3.3.1]{Alperadequate} (see \cite[5.2.9]{Alperadequate}).  From this and \cite[3.4.5]{Alperadequate} it follows  that if $f\in \Gamma (\widetilde X_s, \mls O_{\widetilde X_s})^G$ is an element then there exists $M>0$  such that $f^M$ lifts to an element of $\Gamma (\widetilde X, \mls O_{\widetilde X})^G$.  Now the formation of $G$-invariants commutes with flat base change \cite[5.2.9 (1)]{Alperadequate} so $\Gamma (\widetilde X_s, \mls O_{\widetilde X_s})^G\otimes _{\kappa (s)}\kappa (\bar x)\simeq \Gamma (\widetilde X_{\bar x}, \mls O_{\widetilde X_{\bar x}})^G$.  It follows that there exists $f\in A_{\chi ^m}$ for some $m>0$ such that $\bar x^*f\neq 0$ if and only if this holds for the base change to $\bar x$.  In other words, it suffices to prove (i), and therefore the entire lemma, in the case when $R$ is an algebraically closed field $k$ and $\bar x$ and $\bar y$ are $k$-points.  In what follows we write simply $x$ and $y$ for these $k$-points. 

With this assumption let us prove (i). 
Let $J \subset A[u]^G$ be the ideal of elements without constant term in $u$ (observe that $J$ is equal to the $G$-invariant part of the ideal $A[u] \cdot u \subset A[u]$) and let $\pi: \widetilde{X} \to \Spec(A[u])^G$ be the projection. There exists $m>0$ and $f \in A_{\chi^m}$ such that $x^*f\neq 0$ if and only if there exists $h \in J$ such that $y^*h \neq 0$ (using \eqref{eq:identify}), if and only if $\pi(y)$ is not in $V(J)$. Since $\pi$ is an adequate moduli space and $X \times \{0\}$ surjects onto $V(J)$, it follows from \cite[5.3.1 (4)]{Alperadequate} that $\pi(y)$ is in $V(J)$ if and only if the orbit closure of $y$ meets $X \times \{0\}$. This proves (i).



To prove (ii), note that by (i) we may assume on either side of the desired equivalence that we have some $m>0$ and $f \in A_{\chi^m}$ for which $x^*f \neq 0$. If we set $\widetilde{X}_m:= X \times \mathbf{A}^1_{R, \chi^{-m}}$ then the $m^{th}$-power morphism $\widetilde{X} \to \widetilde{X}_m$ is finite and $G$-equivariant. Writing $y^m$ for the image of $y$, we claim the orbit $O_y$ is closed and the stabilizer $G_y$ is finite if and only if $O_{y^m}$ is closed and $G_{y^m}$ is finite. To see this
consider the maps
\[
G \xrightarrow{\;\;g \mapsto g \cdot y\;\;} O_y \twoheadrightarrow O_{y^m},
\]
which induce identifications $O_y=G/G_y$ and $O_{y^m} = G/G_{y^m}$.  Since $O_y\rightarrow O_{y^m}$ is finite-to-one we see from this that $G_y\subset G_{y^m}$ has finite index, which implies that $G_y$ is finite if and only if $G_{y^m}$ is finite.  Furthermore, if $O_y$ is closed then so is $O_{y^m}$, being the image under a finite morphism of a closed set. 
 Conversely, if $G_{y^m}$ is finite and $O_{y^m}$ is closed, then we have morphisms
\[
G \xrightarrow{\;\;g \mapsto g \cdot y\;\;} \overline{O}_y \twoheadrightarrow O_{y^m}
\]
where the composition is surjective and a finite morphism of schemes. It follows that $G \to \overline{O}_y$ is a dominant proper map and therefore surjective. This completes the proof of the claim.

On the other hand, the function $f^{-1} \in \Gamma(X_f, \mls O^*_{X_f})$ defines a $G$-equivariant section $\sigma: X_f \to (\widetilde{X}_m)_{fu^m}$ of the projection (here we write $X_f$ in place of $D(f)$ to make clear which scheme we are considering). We may even replace $\sigma$ by a $\mathbf{G}_{m, k}$-multiple and assume $\sigma(x) = y^m$. Since $fu^m$ is a $G$-invariant function on $\widetilde{X}$ its value on $O_{y^m}$ is constant, hence $fu^m$ is nonzero on the orbit closure $\overline{O}_{y^m}$; i.e., $\overline{O}_{y^m}$ is contained in $(\widetilde{X}_m)_{fu^m}$. 

Since $\sigma$ is a $G$-equivariant closed embedding, it follows that $G_x$ is finite and $O_x$ is closed in $X_f$ if and only if $G_{y^m}$ is finite and $O_{y^m}$ is closed in $(\widetilde{X}_m)_{fu^m}$. But by the preceeding paragraph $O_{y^m}$ is closed in $(\widetilde{X}_m)_{fu^m}$ if and only if $O_{y^m}$ is closed in $\widetilde{X}$.

\end{proof}

\begin{cor}\label{C:B.16}
With notation as in \ref{D:B.7} let $p: X \to \Spec(R)$ be the projection. Then a geometric point $\bar x \to X$ is stable if and only if the geometric point $\bar x \to X_{\kappa (\bar x)}:= X\times _{\Sp (R)}\bar x$ is stable (as a geometric point of the $G_{\bar x}$-scheme $X_{\kappa (\bar x)}$).    
\end{cor}
\begin{proof}
    This follows from the characterization of stable points in \ref{L:B.15} (ii) (it also follows from the first part of the proof).
\end{proof}

\subsection{Removing the noetherian assumption}\label{S:GIT without noeth}

For technical reasons it is useful to remove the noetherian assumption on $R$.  This assumption is needed to ensure good properties (in particular, finite generation) of the ring of invariants $A^G$  but not for the definition and results about the stack quotient  $[X^s/G]$ as we now explain.

    Let $f:X\rightarrow S$ be a morphism of schemes that is affine and of finite presentation and let $G/S$ be a reductive group scheme over $S$ \cite[XIX, 2.7]{SGA3reductive} acting on $X$. 
Let $\chi: G \to \mathbf{G}_{m, S}$ be a character. 

\begin{defn} A geometric point $\bar x\rightarrow X$ is \emph{stable with respect to} $\chi$ if the corresponding point of the fiber $X\times _S\bar x$ is stable for the action of $G_{f(\bar x)}$ with respect to $\chi_{f(\bar x)}$.
\end{defn}

\begin{rem} In the case when $S = \Sp (R)$ is affine with $R$ noetherian this definition coincides with the notion in \ref{D:B.7} by \ref{C:B.16}.
\end{rem}

\begin{thm}\label{T:GIT} (i) There exists a unique open subscheme $X^{s,\chi}\subset X$ such that a geometric point $\bar x\rightarrow X$ factors through $X^{s,\chi}$ if and only if $\bar x$ is stable.

(ii) The stack quotient $[X^{s,\chi}/G]$ has finite diagonal over $S$.
\end{thm}
\begin{proof}
    The assertions are Zariski local on $S$ so it suffices to consider the case when $S = \Sp (R)$ is affine.  Write $R=\text{colim}_iR_i$ as a colimit of rings of finite type over $\mathbf{Z}$.  Since $f$ is of finite presentation, and also using \cite[XIX, 2.5]{SGA3reductive}, there exists an index $i$, an affine scheme $X_i/R_i$ with action of a reductive group scheme $G_i/R_i$ such that $(X, G)$ is obtained by base change along the map $R_i\rightarrow R$ from $(X_i, G_i)$.  This reduces the proof to the case when $S= \Sp (R)$ with $R$ of finite type over $\mathbf{Z}$ and hence to \ref{L:B.9} and \ref{C:B.12}.
\end{proof}

\end{appendix}

\bibliographystyle{amsplain}
\bibliography{bibliography}{}

\end{document}